\documentclass[10pt,reqno]{article}
\usepackage{amsmath,amssymb,amsthm}
\usepackage{esint}
\usepackage{bm}
\usepackage{subfigure}
\usepackage{graphicx,color}
\usepackage{hyperref}
\hypersetup{
    colorlinks=true,
    linkcolor=blue,
    filecolor=magenta,      
    urlcolor=cyan,
    citecolor=blue,
}

\newtheorem{theorem}{Theorem}[section]
\newtheorem{proposition}[theorem]{Proposition}
\newtheorem{definition}[theorem]{Definition}
\newtheorem{remark}[theorem]{Remark}

\newcommand{\pathfigures}{Figures/}
\graphicspath{{\pathfigures}}
\usepackage{cleveref}

\begin{document}
\title{Fractional calculus approach for the phase dynamics of Josephson
junction under the influence of magnetic field}

\author{
Amer Rasheed\thanks{Department of Mathematics, School of Science and Engineering, Lahore University of Management Sciences, Opposite Sector U, DHA, Lahore Cantt., 54792, Pakistan.}, 
\and
Imtiaz Ali
}
\maketitle

\begin{abstract}
This article presents the phase dynamics of an inline long Josephson junction in voltage state under the influence of constant external magnetic field. Fractional calculus approach is used to model the evolution of the phase difference between the macroscopic wave functions of the two superconductors across the junction. The governing non-linear partial differential equation is then solved using finite difference-finite element schemes. Other quantities of interest like Josephson current density and voltage across the junction are also computed. The effects of various parameters in the model on  phase difference,Josephson current density and voltage
are analyzed graphically with the help of numerical simulations.
\end{abstract}



\section{Introduction}
Superconductors have found a wide range of applications in modern times. They are used in alomst all modern day technologies . Extensive researches have been carried out about superconductors and are still going on. An important phenomenon which arises from superconductivity is Josephson Effect. Two superconductors when weakly connected, for example by a very thin non superconducting barrier, a current flows through the barrier without any dissipation of energy and power supply\cite{r7}. This phenomenon is called Josephson Effect and the junction is called Josephson Junction. A wide range of modern day technologies operate making use of this phenomenon. Superconducting quantum interference devices (SQUIDs) and Superconducting tunnel junction detectors (STJs) are based on Josephson Effect, see for example \cite{r11}. SQUID is a highly sensitive magnetometer which can detect extremely small magnetic fields. Similarly shuted Josephson junctions are used in RSFQ digital electronics \cite{r11}. RSFQ is used to process digital signals.\par
The electrodynamics of the Josepshon junction is completely described by the phase difference of the macroscopic wavefunctions of the two superconductors across the junction\cite{r7}. The evolution of the phase difference is modeled by a nonlinear partial differential equation which is a perturbed Time-Dependent Sine Gordon Equation \cite{r7}. No general analytical solution to this equation has been found  yet\cite{r8}. So various numerical schemes are used to find approximate solution. Different numerical schemes have been used by different researchers to solve the model for various geometries of the junction. The solution of long Josephson junction using finite difference method is discussed by Tinega Abel Kurura, Oduor Okoya and Omolo Ongati \cite{r1}. The steady state case of window Josephson junction using Finite element mehtod is studied by Manolis Vavalis, Mo Muc and Giorgos Sarailidi \cite{r2}. P. Kh. AtanasovaEmail, T. L. Boyadjiev, Yu. M. Shukrinov, E. V. Zemlyanaya \cite{r12} discussed static distributions of the magnetic flux in long Josephson junctions. The Numerical simulations of long Josephson junctions driven by large external radio-frequency signals is analyzed by G. Rotoli, G. Costabile, and R. D. Parmentier \cite{r13}. J. G. Caputo, N. Flytzanis, E. A. Vavalis \cite{r14} discussed a semi-linear elliptic PDE model for the static solution of Josephson junctions. Numerical Study of a system of long Josephson junctions with inductive and capacitive couplings are studied by I. R. Rahmonov , Yu. M. Shukrinov ,  A. Plecenik ,  E. V. Zemlyanaya and M. V. Bashashin \cite{r15}. \par
In all above mentioned researches they have used integer order time derivative in their models. But it has been validated by experiments that the physical models involving continuous evolution are better modeled by employing fractional derivatives. These models keep tracks
and memory of previous deformations which helps to give more accurate physical simulations as compared to integer order time derivatives.\par
In this article the evolution of the phase difference between the macroscopic wave functions of the two superconductors across a long
inline Josephson junction, in voltage state and under the influence of magnetic field, is modeled
using fractional time derivatives instead of integer order time derivatives. The model is then solved using the Finite element scheme along with Finite difference method. Quantities of interest, $i.e$, phase difference, Josephson current density and voltage, are then studied graphically. The effect of various parameters involved in the model, on the aforementioned quantities of interest are then discussed in detail.\par
The mathematical formulation of the model is presented in Section \cref{S2}. In Section \cref{S3} the numerical solution of the model using finite element-finite difference scheme is discussed. Error analysis and convergence of the numerical scheme is presented in Section \cref{S4}. Graphical results and their discussions are given in Section \cref{S5} and conclusion is given in Section \cref{S6}.
  
\section{Mathematical Formulation} \label{S2}
\begin{definition}
The Caputo left sided fractional derivative of order $\alpha\,$, where $\alpha \in \mathbb{C}\,$ and $\,\Re e \{\alpha\} > 0 \,$, of a complex valued function \textit{f(t)}, $\,\textit{t} \in \mathbb{R}\,$, with respect to \textit{t} is defined as  \cite{r5}
$$ \frac{\partial^\alpha}{\partial {\textit{t}}^\alpha} f(t) = \partial^\alpha f(t) = \frac{1}{\Gamma(m-\alpha)}\int_0^t (\,t - \tau\,)^{m-\alpha-1}\frac{\partial^m}{\partial{\tau}^m} f(\tau)\,d\tau$$
where $\textit{m} \in \mathbb{N}$ is such that $\,m-1<\Re e\{\alpha\} < m.\,$\\
Here $\Gamma(\cdot)$ denotes the Euler Gamma function defined as: 

$$ \Gamma(z) = \int_\mathbb{R}\xi^{z-1}e^{-\xi}\,d\xi,\qquad z\in \mathbb{C\quad}, \Re e\{z\}>0$$
\end{definition}

\begin{remark}
If $\alpha \in \mathbb{R}$ then the fractional derivative $\partial^\alpha f(t)$ converges to usual derivative of integer order $\partial^mf(t)$ as $\alpha \rightarrow m \in \mathbb{N},\, where \quad m-1< \alpha < m.$ \cite{r5}
\end{remark}

\begin{proposition}
Let $\alpha \in \mathbb{C}\,$ such that $\Re e\{\alpha\} > 0$ and $\,m-1<\Re e\{\alpha\} < m\,$ for some $m \in \mathbb{N}.\,$ Then for $s \in \mathbb{C}\,$ such that $\Re e\{s\} > 0\,$ we have: \cite{r5}
$$ \partial^{\alpha} t^s = \frac{\Gamma (s+1)}{\Gamma (s+1-\alpha)} t^{s-\alpha},\qquad \forall t > 0 $$
\end{proposition}
\subsection{Governing equations of the model}
In this article we consider a long Josephson junction with length 2\textit{L} .  So $ 2\textit{L} \gg  \lambda_J $, here $\lambda_J $ is the Josephson penetration depth \cite{r7}. The junction lies in the \textit{y}-\textit{z} plane. To make the model simpler and one-dimensional the width $\textit{W}$ of the junction is kept very small as compared to $\, \lambda_J\,$. i.e. $ \textit{W} \ll  \lambda_J $. In this case the phase difference $\varphi$ will vary only along the length of the junction i.e along z-axis \cite{r7}. We have chosen the  origin of the coordinate system at the midpoint of the length of the junction.The barrier thickness is $\textit{d}$ and its area is \textit{A}. A constant external magnetic field $\mathbf{B}^{ex}$ is applied across the junction in the \textit{y}-direction. So $\mathbf{B}^{ex} = B_y^{ex}\,\mathbf{\hat{j}}$. A constant current, applied in the negative \textit{x}-direction, is maintained across the junction . Let $\mathbf{J}$ be the density of the total current flowing through the junction. So $\mathbf{J} = J_{x}\,\mathbf{\hat{i}}$ and we consider $J_{x} > J_c $, here $J_c $ is the maximum Josephson current density. So under these circumstances the Josephson junction resists in the voltage state \cite{r7}. 
In long Josephson junction magnetic field generated by Josephson current cannot be neglected \cite{r7}, so the total magnetic field $\mathbf{B}$ through the junction includes both the externally applied magnetic field and the one generated by Josephson current. For the geometry in Fig. \cref{fig:1.1} \cite{r7} the total magnetic field $\mathbf{B}$ will still point in the \textit{y}-direction and it will vary only along the \textit{z}-direction i.e. $\mathbf{B}(z,t) = B_y(z,t)\,\mathbf{\hat{j}}$ and we can write:  \cite{r7}

\begin{figure}[!h]
\begin{center}
\includegraphics[width = .80\textwidth, height = .62\textwidth]{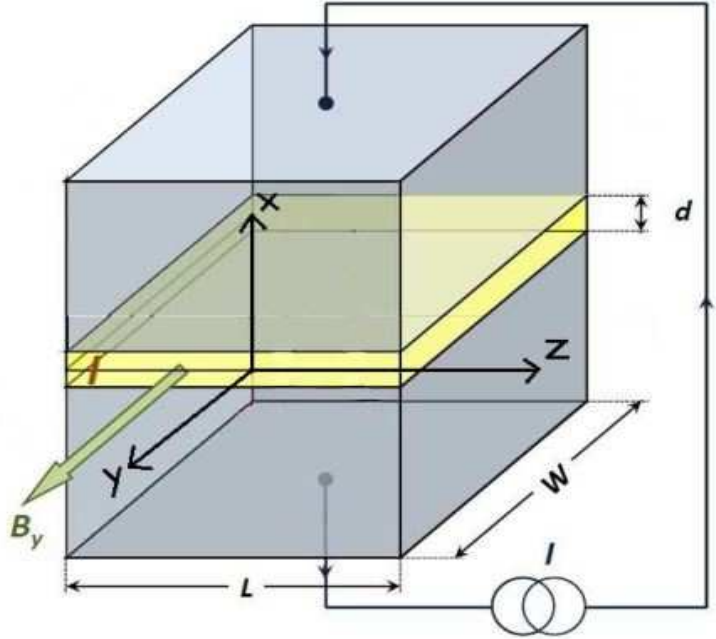}
\caption{Josephson Junction}
 \label{fig:1.1}
\end{center}
\end{figure}

\begin{equation} \label{eq:1}
    \frac{\partial \varphi(z,t)}{\partial z} = \frac{2\pi t_{B} B_y(z,t)}{\Phi_{0}}
\end{equation}
Where $\Phi_{0} = \frac{\displaystyle h}{\displaystyle 2e}$ is the magnetic flux quantum , $t_B = d + \lambda_{L1} + \lambda_{L2}$ $, \, \lambda_{L1} \,\textrm{and}\,\lambda_{L2}\,$ are the London penetration depth of the two superconductors respectively across the junction and $d$ is junction's thickness.\par
Using fractional form, with respect to time $t$, of Ampere's law, given in \cite{r16}-\cite{r17}
\begin{equation} \label{eq:2}
    \nabla \times \mathbf{B} = \mu_0\mathbf{J} + \frac{\mu_0 \epsilon \epsilon_0}{\eta^{1 - \alpha}} \frac{\partial^\alpha \mathbf{E}}{{\partial t}^\alpha}
\end{equation}
where $\mathbf{E}$ is the electric field, $\epsilon$ is the dielectric constant of the junction material and $\eta$ is an arbitrary quantity with dimension of [$\mathbf{T}$]. Since current is flowing only in the \textit{x}-direction, so $\mathbf{E}$ will point in the \textit{x}-direction and will vary along the \textit{z}-direction i.e. $\mathbf{E}(z,t) = {E_x(z,t)}\,\mathbf{\hat{j}}$. So \eqref{eq:2} becomes,
\begin{equation}  \label{eq:3}
    -\frac{\partial B_{y}(z,t)}{\partial z} = \mu_0 J_{x}(z,t) + \mu_0 \epsilon \epsilon_0 \frac{{\partial}^\alpha E_{x}(z,t)}{{\partial t}^\alpha}
\end{equation}
Now \eqref{eq:1} implies,
\begin{equation} \label{eq:4}
        \frac{\partial B_y(z,t)}{\partial z} = \frac{\Phi_{0}}{2\pi t_{B}} \frac{{\partial}^2 \varphi(z,t)}{{\partial z}^2}
\end{equation}
Now let \textit{V(z,t)} be the voltage across the junction. So 
\begin{equation*}
E_{x}(z,t) = -\frac{V(z,t)}{d}.
\end{equation*}
On the same line of argument given in \cite{r16}-\cite{r17} the fractional form, with respect to time $t$, of Josephson second equation(phase-voltage relation) can be written as:
\begin{equation} \label{eq:5}
\frac{1}{{\eta}^{1- \alpha}}\frac{\partial^{\alpha}\varphi(z,t)}{{\partial t}^{\alpha}} = \frac{2\pi}{{\Phi}_0}V(z,t)
\end{equation}
So,
\begin{equation} \label{eq:6}
    \frac{{\partial}^\alpha E_{x}(z,t)}{{\partial t}^\alpha} = -\frac{\Phi_{0}}{2\pi d {\eta}^{1- \alpha}}\frac{{\partial}^{2\alpha} \varphi}{{\partial t}^{2\alpha}}
\end{equation}

Now the total current density $J_{x}$ would have contributions from three different types which are Josephson current(or supercurrent) density $(J_{s})$ , normal current(or resistive current) density $(J_N)$ and bias current density $(J_{bias}).$  \cite{r7} 
\begin{equation} \label{eq:7}
    J_{x} = -J_{s} - J_{N} + J_{bias}   
\end{equation}
Negative sign because the current is flowing in negative \textit{x}-direction.Now,
\begin{equation} \label{eq:8}
    J_{N} = \frac{\sigma_{0} V}{A}= \frac{\sigma_{0} \Phi_{0}}{2\pi A {\eta}^{1- \alpha}}\frac{\partial^{\alpha} \varphi}{{\partial t}^{\alpha}},
\end{equation}
Where $\sigma_{0}$ is the resistivity of the junction. Now from Josephson first equation (phase-current relation) we have,
\begin{equation} \label{eq:9}
J_{s} = J_{c} \sin{(\varphi)}
\end{equation}
 
Substituting  \eqref{eq:8} and \eqref{eq:9} in \eqref{eq:7}, we get
\begin{equation} \label{eq:10}
    J_{x}(z,t) = -J_{c}\sin{(\varphi(z,t))} - \frac{\sigma_{0} \Phi_{0}}{2\pi A {\eta}^{1- \alpha}}\frac{\partial^{\alpha} \varphi}{{\partial t}^{\alpha}} + J_{bias} 
\end{equation}

Substituting \eqref{eq:4},\eqref{eq:6} and  \eqref{eq:10} in \eqref{eq:3} and simplifying, we get,
\begin{equation} \label{eq:11}
        -\frac{{\partial}^2 \varphi(z,t)}{{\partial z}^2} + \frac{1}{{\overline{c}}^2 {\eta}^{1- \alpha}}\frac{{\partial}^{2\alpha} \varphi(z,t)}{{\partial t}^{2\alpha}}
    + \frac{\beta}{{\overline{c}}^2 {\eta}^{1- \alpha}}\frac{{\partial}^{\alpha} \varphi(z,t)}{{\partial t}^{\alpha}}
    + \frac{1}{{\lambda_J}^2}\sin{(\varphi(z,t))}
    = \frac{J_{bias}}{{J_{c} \lambda_J}^2}
\end{equation}
\begin{align*}
\textrm{where}\quad \overline{c} &= \sqrt{\frac{d}{\epsilon \epsilon_{0} \mu_{0} t_{B}}}, \ \ \beta = \frac{\sigma_{0}}{C} \ \ \textrm{and} \ \
{\lambda}_J = \sqrt{\frac{{\Phi}_0}{2\pi{\mu}_0 t_B J_c}} \\
 C &= \frac{\epsilon \epsilon_{0} A}{d} \quad \textrm{is the capacitance of the junction} 
\end{align*}
\subsection{Initial and Boundary Conditions}
Initially the voltage \textit{V} across the junction is 0. So from Josephson second equation we have:
\begin{equation}\label{eq:12}
 \frac{\partial \varphi}{\partial t}(z,0) = 0
\end{equation}
In case of inline junction the Neuman boundary is given by\cite{r7}:
\begin{equation} \label{eq:13}
\frac{\partial \varphi}{\partial z}(-L,t) = \frac{2\pi t_{B}}{\Phi_{0}}\big({B_y}^{ex} -\frac{I\mu_{0}}{2W} \big) \quad \textrm{and}\quad \frac{\partial \varphi}{\partial z}(L,t) = \frac{2\pi t_{B}}{\Phi_{0}}\big({B_y}^{ex} + \frac{I\mu_{0}}{2W} \big)
\end{equation}
where \textit{I} is the total current flowing through the junction.\par
Since $L \gg \lambda_J$. Lets take $ L = 10\lambda_J$. We take following boundary conditions:
\begin{equation}
    \frac{\partial \varphi}{\partial z}(-10\lambda_J,t)  = \frac{0.00189}{\lambda_J}  \quad
    \textrm{and} \quad \frac{\partial \varphi}{\partial z}(10\lambda_J,t) = \frac{0.00163}{\lambda_J}
\end{equation} \label{eq:14}
For $\varphi(z,0)$  we take a particular solution of the stationary Sine-Gordon equation as our initial phase difference given in \cite{r7}:
\begin{equation} \label{eq:15}
    \varphi(z,0) = 4 \arctan\bigg(\exp{\bigg({\frac{ \frac{\displaystyle z}{\lambda_J} - 0.1}{\sqrt{1 - c^2}}}\bigg)}\bigg)
\end{equation}
Here we take $c = 0.9$. So the final model becomes:
\begin{equation}\label{eq:16}
    \begin{cases} \displaystyle
    -\frac{{\displaystyle\partial}^2 \displaystyle\varphi(z,t)}{{\displaystyle\partial \displaystyle z}^2} + \frac{1}{{\displaystyle \overline{c}}^2 {\eta}^{1 - \alpha}}\frac{{\displaystyle \partial}^{2\alpha} \displaystyle \varphi(z,t)}{{\displaystyle \displaystyle \partial \displaystyle t}^{2\alpha}}
    + \frac{\displaystyle \beta}{{\displaystyle \overline{c}}^2 {\eta}^{1 - \alpha}}\frac{\displaystyle {\partial}^{\alpha} \displaystyle \varphi(z,t)}{\displaystyle {{\partial \displaystyle t}^{\alpha}}}
    + \frac{1}{\displaystyle {\lambda_J}^2}\displaystyle \sin{(\displaystyle \varphi(z,t))}
    = \frac{\displaystyle J_{bias}}{\displaystyle {J_{c} \lambda_J}^2} \\
 \qquad \qquad     \frac{\displaystyle \partial \displaystyle \varphi}{\displaystyle \partial \displaystyle z}(-10 \displaystyle \lambda_J,\displaystyle t)  = \frac{\displaystyle A}{\displaystyle \lambda_J} \qquad \frac{\displaystyle \partial \displaystyle  \varphi}{\displaystyle \partial \displaystyle z}(10\displaystyle \lambda_J,\displaystyle t) = \frac{\displaystyle B}{\displaystyle \lambda_J} \\
  \displaystyle \varphi(\displaystyle z,0) = \displaystyle 4 \displaystyle \arctan \displaystyle \bigg(\displaystyle \exp{\displaystyle \bigg({\frac{ \frac{\displaystyle z}{\displaystyle \lambda_J} - \displaystyle 0.1}{\displaystyle \sqrt{1 - c^2}}}\bigg)}\bigg) \ \ \displaystyle \textrm{and} \ \  
\frac{\displaystyle \partial \displaystyle \varphi}{\displaystyle \partial t}(\displaystyle z , 0) = 0
    \end{cases}
\end{equation}
where $A = 0.00189$ and $B = 0.00163$
 \subsection{Non-dimensionalization of the Model}
To non-dimensionalized the IBVP we use following dimensionless quantities:
\begin{equation}\label{eq:17}
\hat{z} = \frac{z}{{\lambda}_J } \ \ ,\ \  \hat{t} = \frac{\overline{c} t}{{\lambda}_J}
\end{equation}
As ${\lambda}_J$ has dimension of [$\mathbf{L}$] and $\overline{c}$ has dimension of [$\mathbf{LT^{-1}}$]
Substituting \eqref{eq:17} in \eqref{eq:16} and ignoring hats, we get following dimensionaless model:
\begin{equation} \label{eq:18}
    \begin{cases} \displaystyle
    -{\frac{ \displaystyle \partial^2 \displaystyle \varphi}{ \displaystyle \partial z^2}} +  {\displaystyle \gamma}_2\frac{ \displaystyle \partial^{2} \displaystyle \varphi}{{ \displaystyle \partial  \displaystyle t}^{2}} +  {\displaystyle \gamma}_1\frac{ \displaystyle \partial \displaystyle \varphi}{{ \displaystyle \partial  \displaystyle t}} +  \displaystyle \sin{( \displaystyle \varphi)} =  \displaystyle \lambda  \\
  \frac{ \displaystyle \partial \displaystyle \varphi}{ \displaystyle \partial  \displaystyle z}(-10,  \displaystyle t ) = A \ \ ,\ \ \frac{ \displaystyle \partial \displaystyle \varphi}{ \displaystyle \partial  \displaystyle z}(10 ,  \displaystyle t ) = B \\
   \displaystyle \varphi(z,0) = 4  \displaystyle \arctan\bigg( \displaystyle \exp{\bigg({\frac{  \displaystyle z - 0.1}{ \displaystyle \sqrt{1 - c^2}}}\bigg)}\bigg)  \ \ , \ \  
\frac{ \displaystyle \partial  \displaystyle \varphi}{ \displaystyle \partial  \displaystyle t}( \displaystyle z , 0) = 0 
    \end{cases}
\end{equation}
where ${\gamma}_1 = \frac{{\displaystyle \beta {\lambda_J}^{2 - \alpha}}}{\displaystyle {(\overline{c})}^{2 + \alpha}} ,\quad {\gamma}_2 = \frac{\displaystyle 1}{\displaystyle {(\overline{c})}^{2(\alpha - 1)} {\lambda_J}^{2(\alpha - 1)} {\eta}^{1 - \alpha}}  \quad \textrm{and} \quad \lambda =  \frac{\displaystyle J_{bias}}{\displaystyle J_c}$
\section{ Finite Difference-Finite Element appprximation scheme} \label{S3}
In this section a numerical scheme is developed to find the approximate solution of the model \eqref{eq:18}. Finite difference method is used to discretize time variable  while space variable is discretized using finite element method.
\subsection{Spaces and Notations}
$\textit{L}^2(\,\Omega\,)$ denotes the space of square integrable  functions over the domain $\Omega = [-10 , 10].\, \textit{L}^2(\,\Omega\,)$ is equipped with following inner product and norm:
\begin{equation} \label{eq:19}
(\,f,g\,)_{\textit{L}^2(\,\Omega\,)} := \int_{\Omega}fg\;dz \quad \textrm{and} \quad ||\,\textit{f}\,||_{\textit{L}^2(\,\Omega\,)} := {\Bigg( \int_\Omega |\,\textit{f}\,|^2 \Bigg)}^\frac{1}{2} dz \quad f,g \in \textit{L}^2(\,\Omega\,) 
\end{equation}
In this article we will denote $(\,f,g\,)_{\textit{L}^2(\,\Omega\,)} \textrm{simply by} (\,f,g\,)$. We also define,
\begin{equation} \label{eq:20}
\langle\, f , g \,\rangle := \bigg(\,\frac{\partial f}{\partial z}\, ,\,\frac{\partial g}{\partial z}\,\bigg)
\end{equation}
$\textit{H\,}^m(\,\Omega\,)$ denotes Sobolev space of order $m > 0$ over the domain $\Omega = [-10 , 10].\, \textit{H\,}^m(\,\Omega\,)$ is equipped with following inner product and norm:
\begin{equation} \label{eq:21}
(\,f\,,\,g\,)_{{\textit{H}\,}^\textit{m}} := \sum_{i = 0}^m \bigg(\,\frac{d^i f}{dz^i}\, ,\,\frac{d^i g}{dz^i}\,\bigg)  \quad \textrm{and} \quad   ||\,\textit{f}\,||_{\textit{H}^m(\,\Omega\,)} :=  \sum_{i = 0}^m \bigg(\,\bigg|\bigg|\,\frac{d^i f}{dz^i}\,\bigg|\bigg|_{\textit{L}^2(\,\Omega\,)}^2\,\bigg)^{\frac{1}{2}} \quad f,g \in \textit{H\,}^m(\,\Omega\,)
\end{equation}

We define following differential operator:
\begin{equation} \label{eq:22}
\mathcal{L}_t^{\alpha}[\xi(.)](t) := \bigg({\displaystyle \gamma}_2\frac{\partial^{2\alpha}}{\partial t^{2\alpha}} + {\displaystyle \gamma}_1\frac{\partial^{\alpha}}{\partial t^{\alpha}}\bigg)[\xi(t)]
\end{equation}

\subsection{Finite Difference approximation}
Let $t_k = \tau k \quad \text{for} \quad k = 0,1,\cdots , m$ and  $\tau = \frac{\displaystyle T}{\displaystyle m} ,\, T > 0 $ , is the time step size. Consider the following approximations of first order time derivatives for $0 \leq k < m$

\begin{equation} \label{eq:23}
 \frac{\partial \varphi}{\partial t}(z,t) \approx  \frac{\partial \varphi}{\partial t}(z,t_{k+1}) \approx \frac{\varphi(z,t_{k+1}) - \varphi(z,t_k)}{\tau} \qquad \text{for}  \quad t_k \leq t \leq t_{k+1}   
\end{equation}

Initial conditions in \eqref{eq:18} yield,
\begin{equation} \label{eq:24}
   \varphi(z , t_0) \approx  4 \arctan\bigg(\exp{\bigg({\frac{ z - 0.1}{\sqrt{1 - c^2}}}\bigg)}\bigg)  \approx    \varphi (z , t_1) \
\end{equation}

Consider the following  approximation of second order time derivative  for $\,  0 \leq k < m-1$
\begin{equation} \label{eq:25}
 \frac{\partial^2 \varphi}{\partial t^2}(z,t) \approx  \frac{\partial^2 \varphi}{\partial t^2}(z,t_{k+2}) \approx \frac{\varphi(z,t_{k+2}) - 2\varphi(z,t_{k+1}) + \varphi(z,t_{k})}{\tau^2} \qquad \text{for}  \quad t_{k+1} \leq t \leq t_{k+2}   
\end{equation}
Fractional order time derivative $\partial_t^\alpha \varphi \,( 0 < \alpha < 1) \, $  can be  approximated using the scheme given by\cite{r9}-\cite{r10}. For $0 \leq k < m$, 
\begin{align} \label{eq:26}
       \frac{\partial^\alpha \varphi}{\partial t^\alpha}(z , t_{k+1}) & = \frac{1}{\Gamma(1 - \alpha)} \sum_{s=0}^{s=k}\int_{s\tau}^{(s+1)\tau}\frac{\partial \varphi(z,\omega)}{\partial \omega}\frac{1}{(t_{k+1}-\omega)^\alpha}\,d \omega \nonumber  \\
       &= \frac{1}{\Gamma(1 - \alpha)}\sum_{s=0}^{s=k}\frac{\varphi(z , t_{s+1}) - \varphi(z , t_{s})}{\tau}\int_{s\tau}^{(s+1)\tau}\frac{1}{(t_{k+1}-\omega)^\alpha}\,d \omega \nonumber\\
       &= C_\alpha\big[\varphi(z , t_{k+1}) - \varphi(z , t_{k})\big] + C_\alpha \zeta_k^\alpha[\varphi]  
\end{align}
\begin{align} \label{eq:27}
   \textrm{where} \quad \zeta_k^\alpha[\varphi] &= \sum_{p=1}^{p=k}\big[\varphi(z , t_{k-p+1}) - \varphi(z , t_{k-p})\big]\,b_p^\alpha \qquad \quad \textrm{for} \quad 0 \leq k < m  \\
  \zeta_0^\alpha[\varphi] := 0, \quad C_\alpha &= \frac{\tau^{-\alpha}}{\Gamma(2 - \alpha)} \quad \textrm{and}  \quad  b_p^\alpha =  \big[(p+1)^{1-\alpha} - p^{1-\alpha}\big]   \nonumber
\end{align}
\eqref{eq:24} implies that $\zeta_1^\alpha[\varphi] = 0\,.$ \par
Similarly fractional order time derivative $\partial_t^{2\alpha} \varphi \,( 0 < \alpha < 1) \, $ can be approximated again using the scheme given in \cite{r9}-\cite{r10}. So for $0 \leq k < m-1$, 
 \begin{align} 
    \frac{\partial^{2\alpha} \varphi}{\partial t^{2\alpha}}(z , t_{k+2}) &= C_{2\alpha} \big[\varphi(z , t_{k+2}) - 2\varphi(z , t_{k+1}) + \varphi(z , t_{k})\big] + C_{2\alpha} \zeta_k^{2\alpha}[\varphi] \label{eq:28}  \\
\textrm{where} \quad \zeta_k^{2\alpha}[\varphi] &= \sum_{p=1}^{p=k}\big[\varphi(z , t_{k-p+2}) - 2\varphi(z , t_{k-p+1})+ \varphi(z , t_{k-p})\big]\,b_p^{2\alpha} \quad  0 \leq k < m-1  \label{eq:29} \\
 \zeta_0^{2\alpha}[\varphi] := 0, \quad C_{2\alpha} &= \frac{\tau^{-2\alpha}}{\Gamma(3 - 2\alpha)} \quad \textrm{and} \quad b_p^{2\alpha} =  \big[(p+1)^{2(1-\alpha)} - p^{2(1-\alpha)}\big]  \nonumber
 \end{align}

Using \eqref{eq:26} and \eqref{eq:28} in \eqref{eq:22} $\mathcal{L}_t^{\alpha}[\varphi](t)$ can be approximated as (for $ 0 \leq k < m-1$):
 \begin{align}
   \mathcal{L}_t^{\alpha}[\varphi](t_{k+2}) &= \bigg({\displaystyle \gamma}_1\frac{\partial^{2\alpha}}{\partial t^{2\alpha}} + {\displaystyle \gamma}_1\frac{\partial^{\alpha}}{\partial t^{\alpha}}\bigg)[\varphi](t_{k+2}) \nonumber \\
   &\approx {\displaystyle \gamma}_1 C_\alpha \big[\varphi(t_{k+2}) - \varphi(t_{k+1})\big] + {\displaystyle \gamma}_1 C_\alpha \zeta_k^\alpha[\varphi] 
   +   {\displaystyle \gamma}_2 C_{2\alpha} \big[\varphi(t_{k+2}) - 2\varphi(t_{k+1}) + \varphi(t_{k})\big] \label{eq:30} \\
   &+ {\displaystyle \gamma}_2 C_{2\alpha} \zeta_k^{2\alpha}[\varphi] \nonumber
 \end{align}
 
\subsection{Finite Element Discretization}
Form a partition of the domain $\Omega = (-10,10)$ with $n$ sub-domains $\Omega_i =  (z_{i-1} \,, \,z_{i}) \quad \text{for} \quad i = 1,\cdots , n$ such that
$$-10 = z_0 < z_1 < \cdots < z_n=10$$
 $$\textrm{i.e.} \bigcup_{i=1}^n \overline{\Omega_i} = \overline{\Omega} \quad \textrm{and} \quad \Omega_i \bigcap \Omega_j = \emptyset \quad \forall i,j = 1,2,\cdots , n  \quad \textrm{and} \quad i\neq j$$
We consider sub-domains $\Omega_i$ of equal length, $\,h = \frac{20}{\displaystyle n}\,$
We define a finite dimensional subspace $V_h^1(\Omega)$ of $H^1(\Omega)$ as follows:
 $$ V_h^1(\Omega) = \{v_h \in  H^1(\Omega) : v_h\big |_{\Omega_i} \in \mathbb{P}_r(\Omega_i) ,\, \forall \, i=1,\cdots,n   \}$$
 where $\,\mathbb{P}_r(\Omega_i)\,$ represents space of Lagrange Polynomials of degree at most r over each sub-domain $\Omega_i \quad \forall \, i = 1,\cdots , n$.\par
$\textbf{Weak Formulation}$. Find $\; \varphi \in C^1\big([0,T];H^1(\Omega)\big)\,$ such that: \cite{r5}
\begin{equation} \label{eq:31}  
        \begin{cases} \displaystyle
            \displaystyle \langle \displaystyle\psi,                  \displaystyle \varphi  \displaystyle \rangle +  \displaystyle \mathcal{L}_t^{\alpha}( \displaystyle \psi, \displaystyle \varphi) + ( \displaystyle \psi, \displaystyle \sin{ \displaystyle \varphi}) = ( \displaystyle \psi, \displaystyle \lambda) - A \displaystyle \psi(-10) + B \displaystyle \psi(10) \\
   \displaystyle \varphi(z,0) = 4  \displaystyle \arctan\bigg( \displaystyle \exp{\bigg({\frac{ z - 0.1}{ \displaystyle \sqrt{1 - c^2}}}\bigg)}\bigg) \qquad  
\frac{ \displaystyle \partial  \displaystyle \varphi}{ \displaystyle \partial  \displaystyle t}(z , 0) = 0 
        \end{cases}
\end{equation}
for all $\,,\psi \in H^1(\Omega)$ \par
Let $\; \varphi_h$ represents an  approximate solution of $\varphi$ in $C^1\big([0,T];V_h^1(\Omega)\big)\,$\par
$\textbf{Discrete Weak Formulation}$. Find $\; \varphi_h \in C^1\big([0,T];V_h^1(\Omega)\big)\,$ such that:  \cite{r5}
 \begin{equation} \label{eq:32}  
        \begin{cases} \displaystyle 
     \displaystyle \langle  \displaystyle \psi_h, \displaystyle \varphi_h  \displaystyle \rangle +  \displaystyle \mathcal{L}_t^{\alpha}( \displaystyle \psi_h, \displaystyle \varphi_h)  + ( \displaystyle \psi_h, \displaystyle \sin{ \displaystyle \varphi_h}) = ( \displaystyle \psi_h, \displaystyle \lambda) - A \displaystyle \psi_h(-10) + B \displaystyle \psi_h(10)\\
   \displaystyle \varphi_h(z,0) = 4  \displaystyle \arctan\bigg( \displaystyle \exp{\bigg({\frac{ z - 0.1}{ \displaystyle \sqrt{1 - c^2}}}\bigg)}\bigg) \quad  
\frac{ \displaystyle \partial  \displaystyle \varphi_h}{ \displaystyle \partial  \displaystyle t}(z , 0) = 0 
            \end{cases}
 \end{equation} 
for all $\,\psi_h \in V_h^1(\Omega)$.\par
Using the approximations of $\mathcal{L}_t^{\alpha}[\varphi](t)$  from \eqref{eq:30} in \eqref{eq:32} for a particular time $t_{k+2} \,(0\leq k < m-1)$ the discrete weak formulation takes the following form: \par
Find $\varphi_h(y,t_{k+2}) \in V_h^1(\Omega)\,$ such that  \cite{r5}:
 \begin{equation} \label{eq:33} 
        \begin{cases} \displaystyle
     \displaystyle \langle  \displaystyle \psi_h, \displaystyle \varphi_h(z,t_{k+2})  \displaystyle \rangle +  \displaystyle \mathcal{L}_{k+2}^{\alpha}( \displaystyle \psi_h, \displaystyle \varphi_h(z,t_{k+2}))  + ( \displaystyle \psi_h, \displaystyle \sin{ \displaystyle \varphi_h(z,t_{k+2})}) \\
    = ( \displaystyle \psi_h, \displaystyle \lambda) - A \displaystyle \psi_h(-10) + B \displaystyle \psi_h(10) \\
      \varphi_h(z , t_0) = 4 \arctan\bigg(\exp{\bigg({\displaystyle \frac{ z - 0.1}{ \displaystyle \sqrt{1 - c^2}}}\bigg)}\bigg)  =    \varphi_h (z , t_1)
            \end{cases}
 \end{equation}
for all $\,\psi_h \in V_h^1(\Omega)$.
Here $\mathcal{L}_{k+2}^{\alpha}( \displaystyle \psi_h, \displaystyle \varphi_h(z,t_{k+2}))$ represents $\mathcal{L}_{t}^{\alpha}[ (\displaystyle \psi_h, \displaystyle \varphi_h)](t_{k+1})$.\par
$\,\varphi_h(z,t)\,$ is defined as follows:
\begin{equation} \label{eq:34} 
    \varphi_h(z,t) = \sum_{p=1}^{D_h} \varphi_p^h(t) \phi_p^h (z) \quad \textrm{where} \quad z \in {\overline{\Omega}} \quad \varphi_p^h(t) = \varphi_h(z_p,t) 
\end{equation}
where $\{ \phi_p^h(z) \big | p=1,\cdots,n\}$ forms a basis of $V_h^1$ and $D_h = \textrm{dim}(V_h^1) $. By choosing $\psi_h$ as $\phi_p^h$ for different values of $ p=1,\cdots,D_h $ in \eqref{eq:33}, following system of non-linear equations is obtained for each $0 \leq k < m-1$ :
\begin{equation} \label{eq:35} 
        \begin{cases} \displaystyle
    \displaystyle \mathbb{A}^h[\displaystyle \varphi_{k+2}^h] + \displaystyle \mathbb{M}^h  \displaystyle \mathcal{L}_{k+2}([ \displaystyle \varphi_{k+2}^h]) +  \displaystyle \mathbb{G}^h( \displaystyle \varphi_{k+2}^h) =  \displaystyle  \mathbb{F}^h\\
    [ \displaystyle \varphi_0^h] = 4 \displaystyle  \arctan\bigg( \displaystyle \exp{\bigg({\frac{ z - 0.1}{ \displaystyle \sqrt{1 - c^2}}}\bigg)}\bigg) = [ \displaystyle \varphi_1^h]
         \end{cases}
 \end{equation}
Where $\, \forall \, p,q = 1,\cdots,D_h$,
\begin{eqnarray*}
  (\mathbb{A}^h)_{pq} &=& (\phi_p^h , \phi_q^h)  = 
    \frac{1}{h} \begin{pmatrix}  
   1 & -1\\
   -1 & 1 
  \end{pmatrix}\\
  (\mathbb{M}^h)_{pq}  &=& \langle  \phi_p^h , \phi_q^h  \rangle  = \frac{h}{6} \begin{pmatrix}
   2 & \, 1\\
   1 & \,2 
  \end{pmatrix}\\
  (\mathbb{G}^h(\varphi_{k+1}^h))_p &=& ( \sin(\varphi^h),\phi_p^h)  = \frac{h}{\varphi_{i+1}^h - \varphi_{i}^h}\begin{pmatrix}
  \cos(\varphi_i^h) - \frac{\displaystyle \sin(\varphi_{i+1}^h)}{\displaystyle \varphi_{i+1}^h + \varphi_{i}^h} + \frac{\displaystyle \sin(\varphi_{i}^h)}{\displaystyle \varphi_{i+1}^h - \varphi_{i}^h} 
  \\
   -\cos(\varphi_{i+1}^h) + \frac{\displaystyle \sin(\varphi_{i+1}^h)}{\displaystyle \varphi_{i+1}^h - \varphi_{i}^h} - \frac{\displaystyle \sin(\varphi_{i}^h)}{\displaystyle \varphi_{i+1}^h - \varphi_{i}^h} 
  \end{pmatrix}\\
 (\mathbb{F}^h)_p  &=& (\lambda , \phi_p^h) = \frac{\lambda h}{2}\begin{pmatrix}
   1\\
   1
  \end{pmatrix} \quad \textrm{for} \quad \textrm{for} \quad 1 < p < D_h\\
   (\mathbb{F}^h)_1  &=&  \begin{pmatrix}
   \frac{\displaystyle \lambda h}{\displaystyle 2} - A \\
   \frac{\displaystyle \lambda h}{\displaystyle 2} 
  \end{pmatrix} \quad \textrm{for} \quad p = 1 \\
     (\mathbb{F}^h)_{D_h}  &=& \begin{pmatrix}
   \frac{\displaystyle \lambda h}{\displaystyle 2} \\
   \frac{\displaystyle \lambda h}{\displaystyle 2}  + B
  \end{pmatrix} \quad \textrm{for} \quad p = D_h \\
\end{eqnarray*} 
 A code has been developed in MatLab R2015a to solve the system \eqref{eq:35} 
\section{Error Analysis and Scheme convergence} \label{S4}
In this section, the convergence and the numerical error analysis of the scheme developed in previous section are discussed. The theoretical predication for the error is then compared with the one obtained using the numerical scheme. The two are found in good agreement with each other thus validating the numerical scheme developed in previous section. The theoretical error estimates of the problems of type under consideration in this article is demonstrated by following theorem .
\subsection{Theoretical Error Predictions}
\begin{theorem}
If $\;\varphi(z,t)\;$ is exact solution of \eqref{eq:18} and $\;\varphi_h(z,t_{k+1})\; $ is solution of the problem \eqref{eq:35}, which is approximate solution of \eqref{eq:18} at time step $ t_{k+1}$, then there exist a constant $C > 0$, independent of space step size h and time step size $\tau$ such that:\\
$
\big\|\varphi(z,t_{k+1}) - \varphi_h(z,t_{k+1})\big\|_{L^2(\Omega)} \leq C(h^{r+1} + \tau^2)\qquad  0 \leq k < m \\
\big\|\varphi(z,t_{k+1}) - \varphi_h(z,t_{k+1})\big\|_{H^1(\Omega)} \leq C(h^{r} + \tau^2)\qquad \quad 0 \leq k < m \\
\textrm{where r is the degree of the Lagrange polynomials used as the basis for} \, V^h(\Omega)$
\end{theorem}
\begin{proof}
    The proof of above error estimates can be derived by using similar reasonong given in \cite{r9}- \cite{r10}. 
\end{proof}
For the model under consideration in this article, linear Lagrange polynomials are used as the basis functions of the space $V_h^1(\Omega)$, so here r = 1.
\begin{figure}[!h]
\begin{center}
\subfigure[$L_2(\Omega)$ and $H^1(\Omega)$ error
     curves]
{
\includegraphics[width = .48\textwidth, height = .42\textwidth]{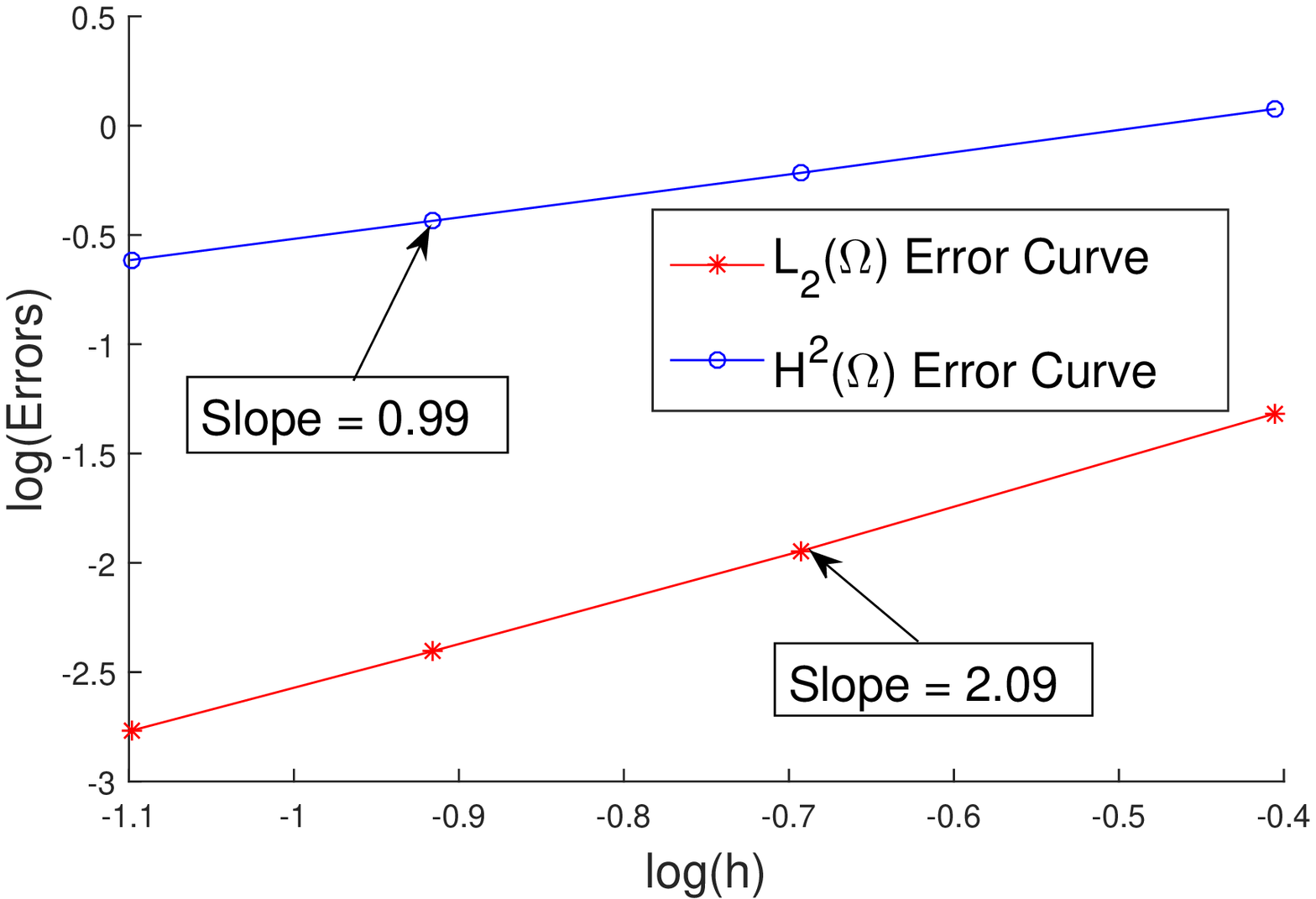}
\label{fig:6.1(a)} }
\subfigure[Fabricated and approximate solution
    curves]
{
\includegraphics[width = .48\textwidth, height = .42\textwidth]{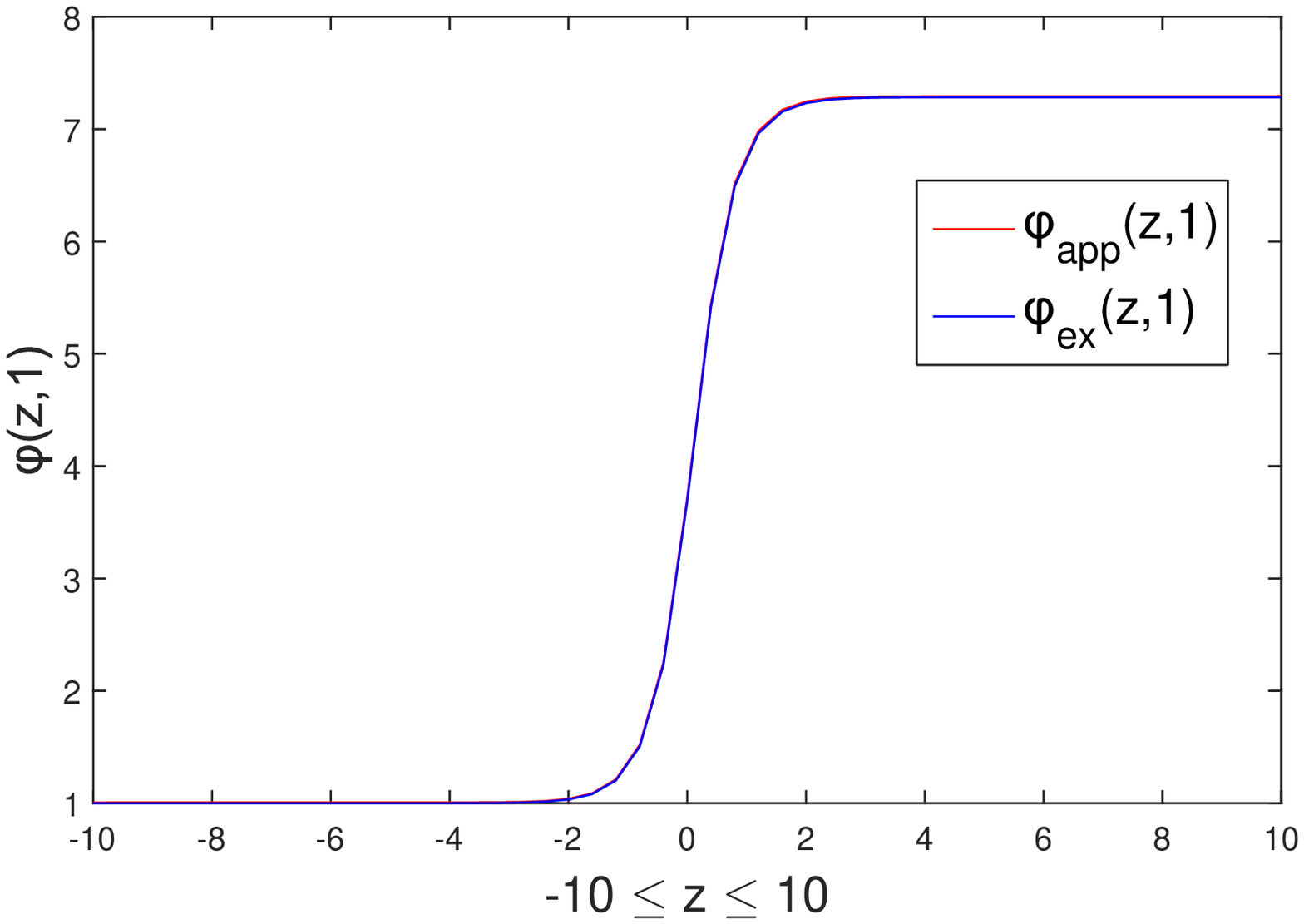}
\label{fig:6.1(b)}}
\caption{Validation of Numerical Scheme}
 \label{fig:6.1}
\end{center}
\end{figure}
\begin{figure}
\begin{center}
\subfigure[Surface plot of approximate solution]
{
\includegraphics[width = .48\textwidth, height = .42\textwidth]{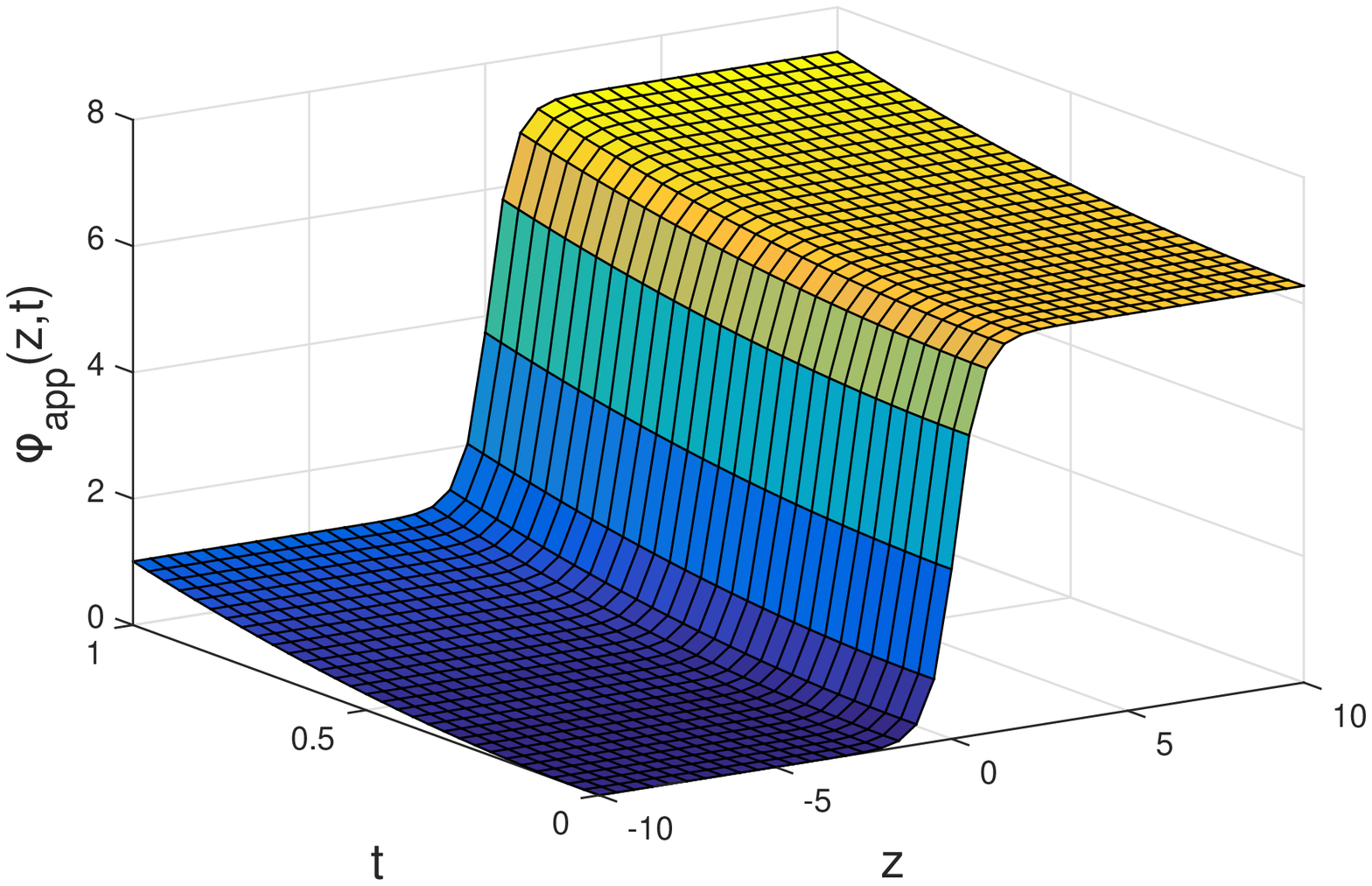}
\label{fig:6.2(a)} }
\subfigure[Surface plot of fabricated solution]
{
\includegraphics[width = .48\textwidth, height = .42\textwidth]{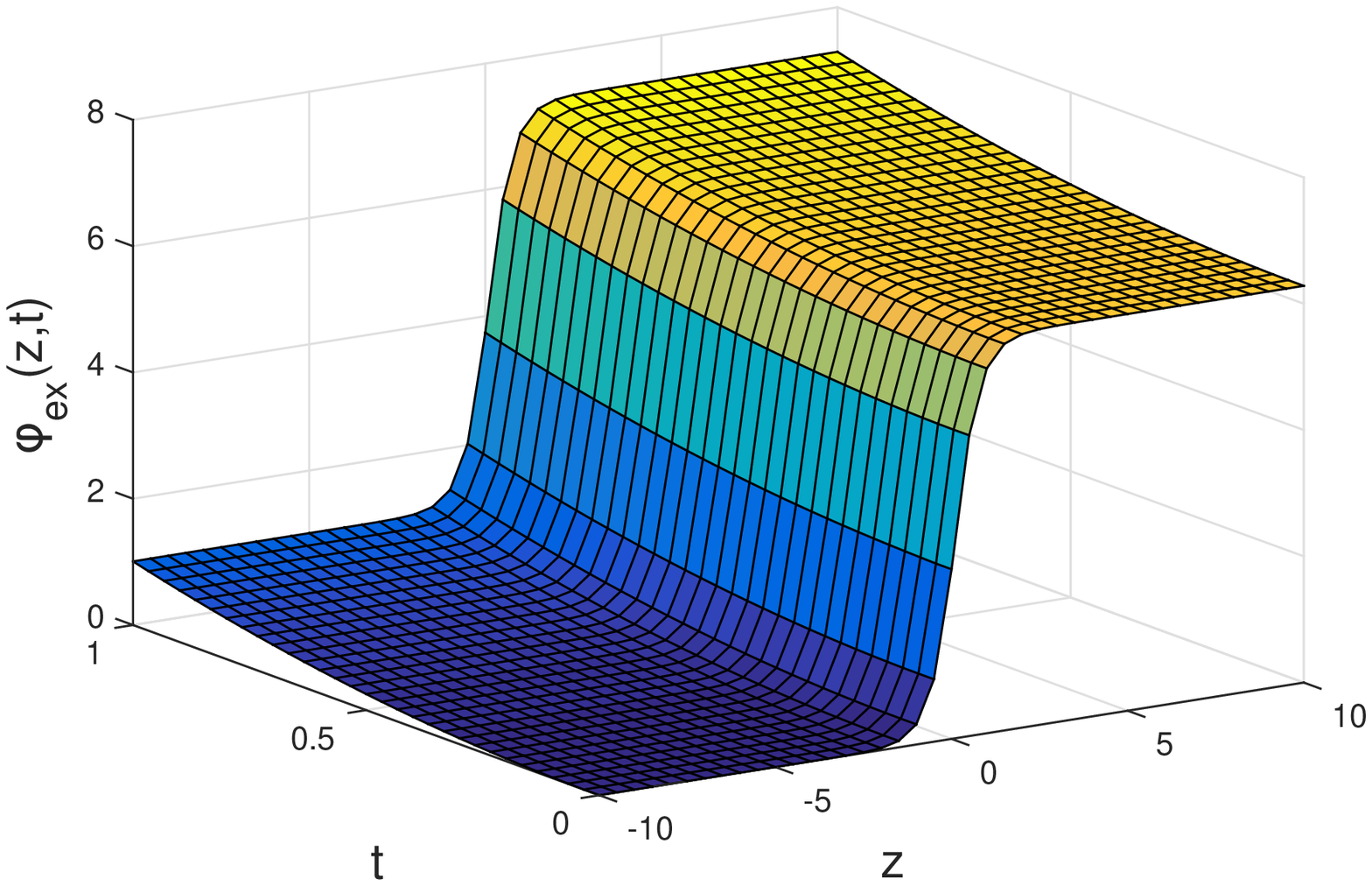}
\label{fig:6.2(b)}}
\caption{Surface plots of approximate and fabricated solution}
 \label{fig:6.2}
\end{center}
\end{figure}
\subsection{Numerical Error Estimates and Scheme validation}

To compare the theoretical and numerical error estimates of the scheme developed in the previous section, an exact solution of the model \eqref{eq:18} is fabricated by adding an artificial term $\,F_{art}(z,t)\,$ on the left side of \eqref{eq:18} \cite{r5}. So we get:
\begin{equation} \label{eq:36}
    \begin{cases} \displaystyle 
    -{\frac{ \displaystyle \partial^2 \displaystyle \varphi}{ \displaystyle \partial  \displaystyle z^2}} + {\displaystyle \gamma}_2 \frac{ \displaystyle \partial^{2 \alpha} \displaystyle \varphi}{{ \displaystyle \partial  \displaystyle t}^{2 \alpha}} + {\displaystyle \gamma}_1\frac{ \displaystyle \partial^{  \alpha} \displaystyle \varphi}{{ \displaystyle \partial  \displaystyle t}^{  \alpha}} +  \displaystyle \sin \displaystyle \varphi =  \displaystyle \lambda +  \displaystyle F_{art}(z,t)  \\
  \frac{ \displaystyle \partial \displaystyle \varphi}{ \displaystyle \partial  \displaystyle z}(-10 ,  \displaystyle t ) = A \qquad \qquad \frac{ \displaystyle \partial \displaystyle \varphi}{ \displaystyle \partial z}(10, t ) = B \\
   \displaystyle \varphi(z,0) = 4  \displaystyle \arctan\bigg( \displaystyle \exp{\bigg({\frac{ z - 0.1}{ \displaystyle \sqrt{1 - c^2}}}\bigg)}\bigg) \qquad  
\frac{ \displaystyle \partial  \displaystyle \varphi}{ \displaystyle \partial  \displaystyle t}(z , 0) = 0
    \end{cases}
\end{equation}
The fabricated exact solution $\,\varphi_{ex}(z,t)\,$  should be smooth and satisfy the IC's and BC's of \eqref{eq:18}. We choose $\,\varphi_{ex}(z,t)\,$  to be: 
\begin{equation*}
\,\varphi_{ex}(z,t) = 4\arctan\bigg(\exp{\bigg({\frac{ z - 0.1}{\sqrt{1 - c^2}}}\bigg)}\bigg) + t^2\,
\end{equation*}
Substituting $\,\varphi_{ex}(z,t)\,$ in \eqref{eq:36} and then solving for $\,F_{art}(z,t)\,$, we get:
{ \small
\begin{align} 
F_{art}(z,t) &= \displaystyle \sin{\bigg(4\arctan \bigg(\exp{\bigg (\frac{x - 0.1}{\sqrt{1 - c^2}}\bigg)}\bigg)+ t^2\bigg )} - 4\Bigg(\frac{ \displaystyle \exp{\bigg(\frac{ \displaystyle x - 0.1}{ \displaystyle \sqrt{1 - c^2}}\bigg)}}{(1 - c^2)\bigg(\exp{(\frac{ \displaystyle (2(x - 0.1))}{ \displaystyle \sqrt{(1 - c^2)}}}\big) + 1\bigg)} \nonumber\\
 &- \frac{2\exp{\bigg (\frac{ \displaystyle 3(x - 0.1)}{ \displaystyle \sqrt{1 - c^2}}\bigg)}}{(1-c^2)\bigg( \displaystyle \exp{\bigg (\frac{ \displaystyle 2(x - 0.1)}{ \displaystyle \sqrt{1 - c^2}}\bigg)} + 1\bigg)^2}\Bigg) + \frac{2 \displaystyle {\displaystyle \gamma}_1  \displaystyle t^{  (2-  \alpha)}}{ \displaystyle \Gamma(3 -   \alpha)} + \frac{2 {\displaystyle \gamma}_2 \displaystyle t^{ (2-2  \alpha)}}{ \displaystyle \Gamma(3 - 2  \alpha)} \label{eq:37}
\end{align}
}

Approximate solution of \eqref{eq:36} is then obtained using the same numerical scheme discussed in the previous section. Errors in the approximate solution of \eqref{eq:36} are calculated in $\,L^2(\Omega)\,$ and $\,H^1(\Omega)\,$ norms using four different values of space step size \textit{h}. These errors are then converted to logarithmic scale and then plotted against $\,\log(h) \,$ in Fig. \cref{fig:6.1(a)}. The slopes of the $\,L^2(\Omega)\,$  and $\,H^1(\Omega)\,$ error curves turns out be 2.02 and 0.98 respectively. According to theoretical predictions given by above mentioned theorem, the slopes of these two error curves should be (r+1) and r respectively\cite{r5}. Since here linear polynomials are used as basis functions , so the slopes should be 2 and 1 respectively. So the theoretical error estimate agrees well with the numerical one which  demonstrates that the numerical scheme developed in the previous section is convergent with the theoretically predicted convergent rates.The approximate and exact solutions of \eqref{eq:36} for the final time (t = 1)  are plotted in Fig. \cref{fig:6.1(b)} . The two curves agrees very well. Also the surface plots of the approximate and exact solutions plotted in Fig. \cref{fig:6.2}, agrees very well. All these computations validates the numerical scheme developed in the previous section.

\section{Simulations and discussions} \label{S5}
In this section the effects of various non-dimensional parameters $\,\lambda \,$ ,$\,{\gamma}_1\, $,$\,{\gamma}_2\, $ and $\,\alpha\,$ on phase difference $\,\varphi\,$ , supper current density $\,J_s\,$ and voltage $\,V\,$ are analyzed in detail with graphical representations. The effects of various parameters on the quantities of interest discussed above are plotted for the final time $\,t = 1\,$ in subsequent figures. In Fig. \cref{fig:7.1} the initial phase difference (scaled by 2$\pi$) and initial supper current density (normalized by maximum Josephson current density $\,(J_{c})\,$) are 
displayed.
In Fig. \cref{fig:7.2} the effect of fractional exponent $\,\alpha\,$ on the phase difference, super current density and voltage are plotted. It is observed that with increase in $\,\alpha\,$ phase difference, super current density and voltage decreases.
Effect of parameter $\,{\gamma}_1\,$ on  phase difference, super current density and voltage is displayed in Fig. \cref{fig:7.3} . The parameter $\,{\gamma}_1\,$ is related to the internal properties of the junction and its structure, like conductivity, capacitance, dielectric constant, London penetration depth, width etc. It is inspected that phase difference , voltage and super current density decrease as $\,\gamma\,$ increases. Since ${\gamma}_1$ is directly related to conductivity or equivalently inversely to resistance, so increase in ${\gamma}_1$  means increase in conductivity or decrease in resistance of the junction.So voltage will decrease by $V = IR$. 
In Fig. \cref{fig:7.4} phase difference, super current density and voltage are plotted for different values of the parameter ${\gamma}_2$. The parameter ${\gamma}_2$ was got introduced due to the use of fractional time derivative. So it is observed that the phase difference, super current density and voltage decrease with increase in ${\gamma}_2$.
Similarly in Fig. \cref{fig:7.5} phase difference, super current density and voltage are plotted for different values of the parameter $\,\lambda\,$. The parameter $\,\lambda\,$ is related to the physical bias current flowing through the junction and the maximum Josephson current density. It is noticed that  phase difference, super current density and voltage all increase as $\,\lambda\,$ increases.  Since $\lambda$ is directly related to the bias current, so increase in $\lambda$ means increase in bias current. Hence voltage will increase as $V = IR$.
In Fig. \cref{fig:7.6} the evolution of phase difference over the time interval [0,1] is displaced for two different sets of parameter values.
\begin{figure}[!h]
\begin{center}
\subfigure[Initial Phase difference]
{
\includegraphics[width = .45\textwidth, height = .38\textwidth]{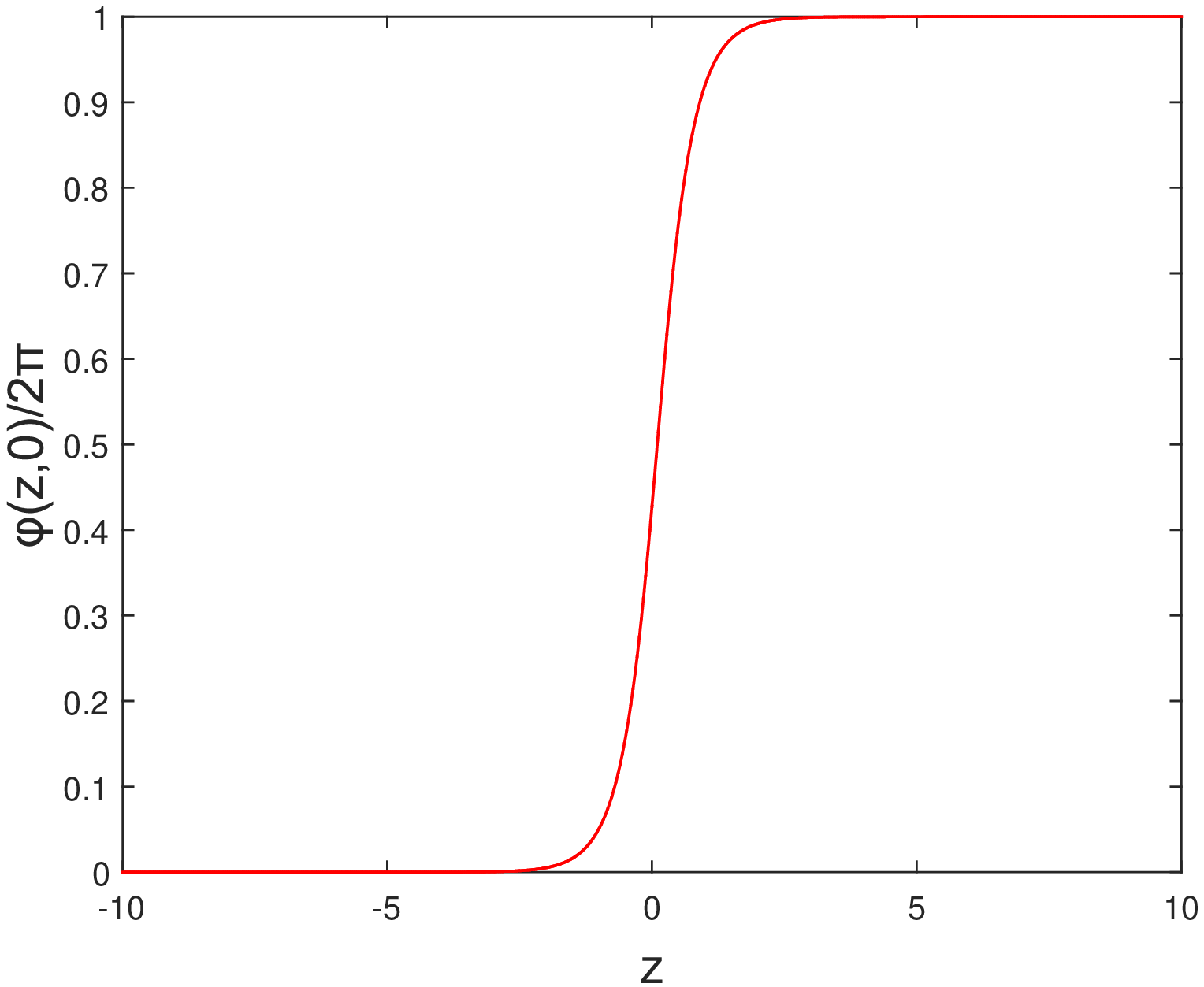}
\label{fig:7.1(a)} }
\subfigure[Initial Josephson current density]
{
\includegraphics[width = .45\textwidth, height = .38\textwidth]{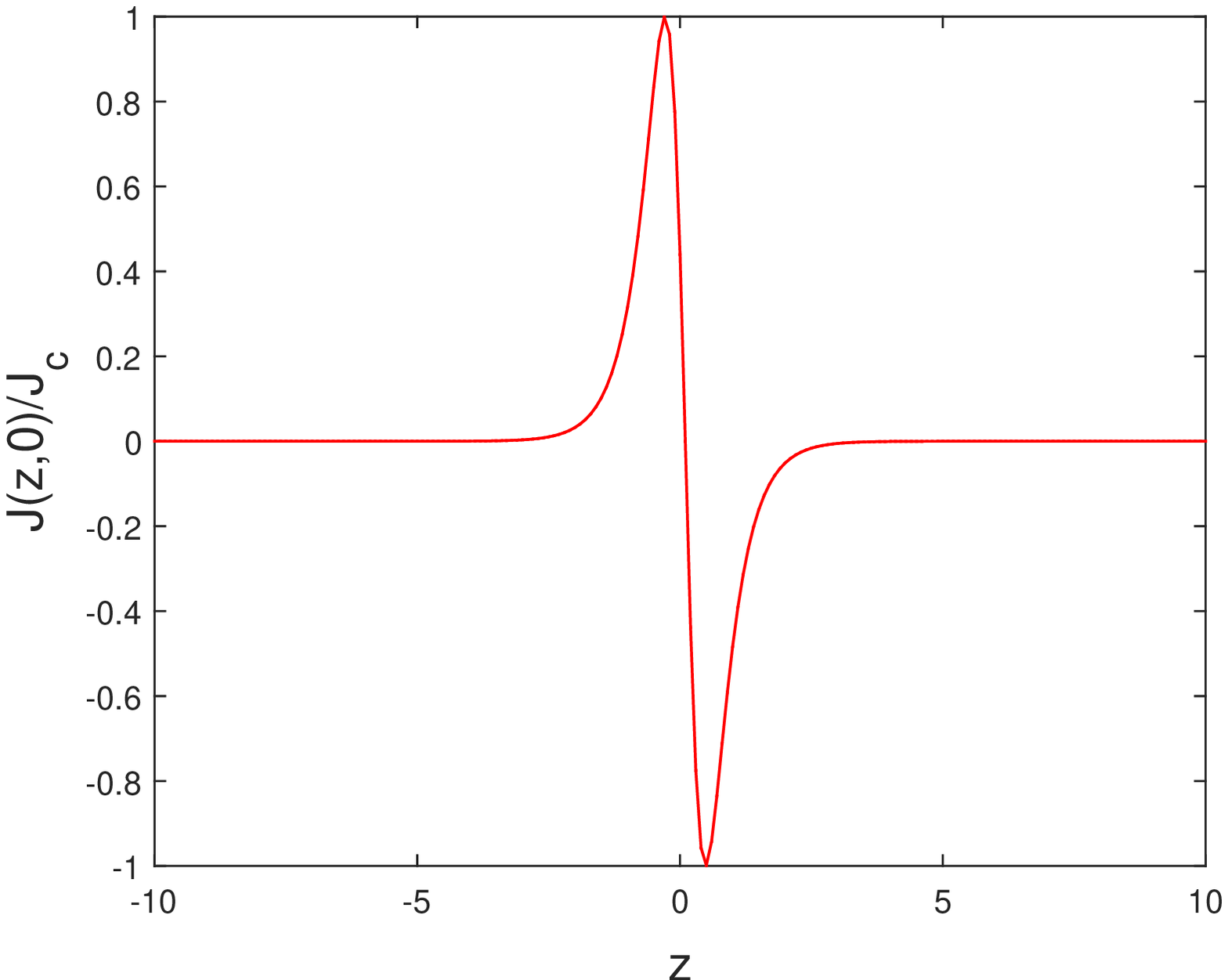}
\label{fig:7.1(b)}}
\caption{Initial Phase difference and Josephson current density}
 \label{fig:7.1}
\end{center}
\begin{center}
\subfigure[Variation of Phase difference
    with $\alpha$]
{
\includegraphics[width = .45\textwidth, height = .38\textwidth]{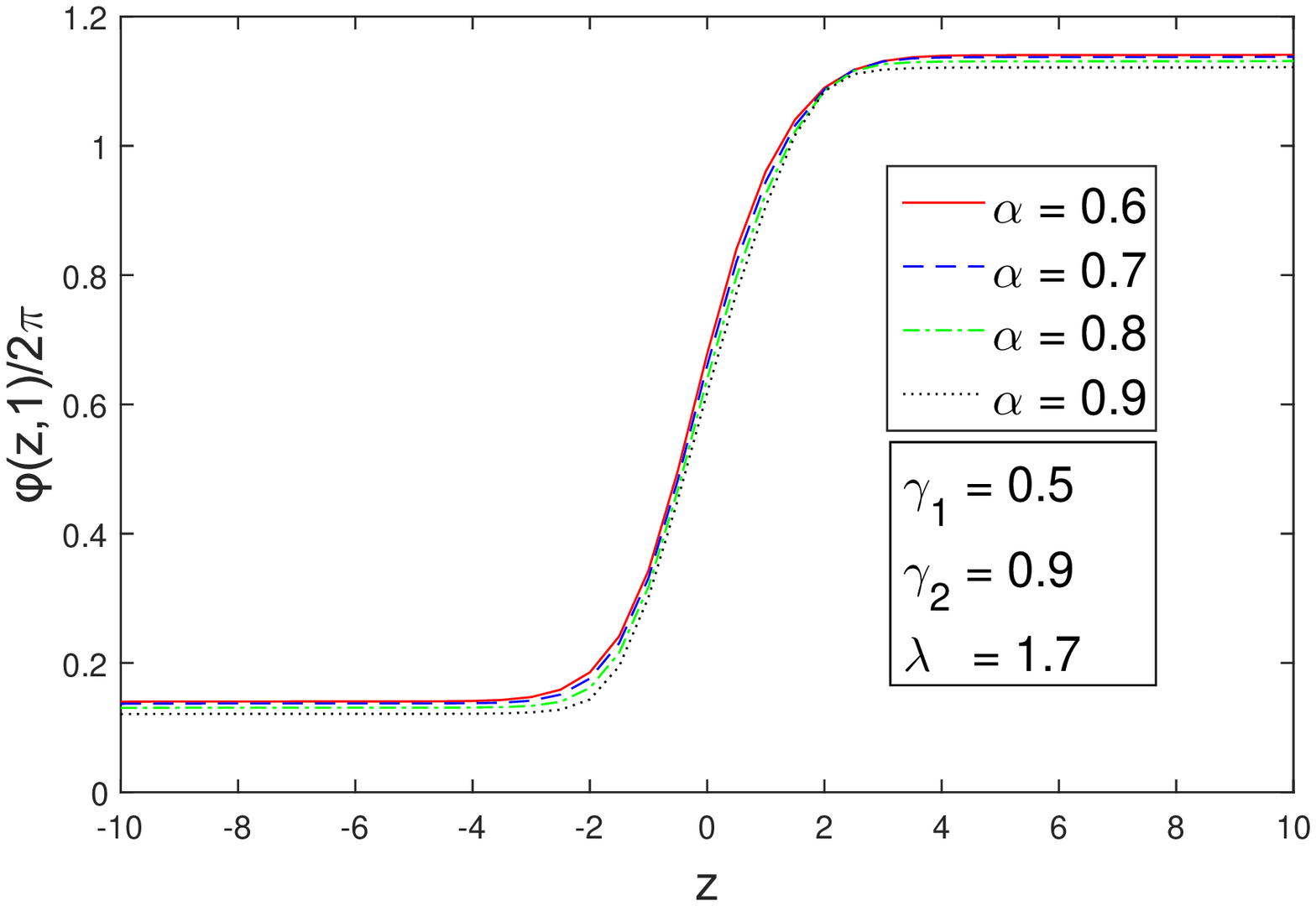}
\label{fig:7.2(a)} }
\subfigure[Variation of Josephson current density with $\alpha$]
{
\includegraphics[width = .45\textwidth, height = .38\textwidth]{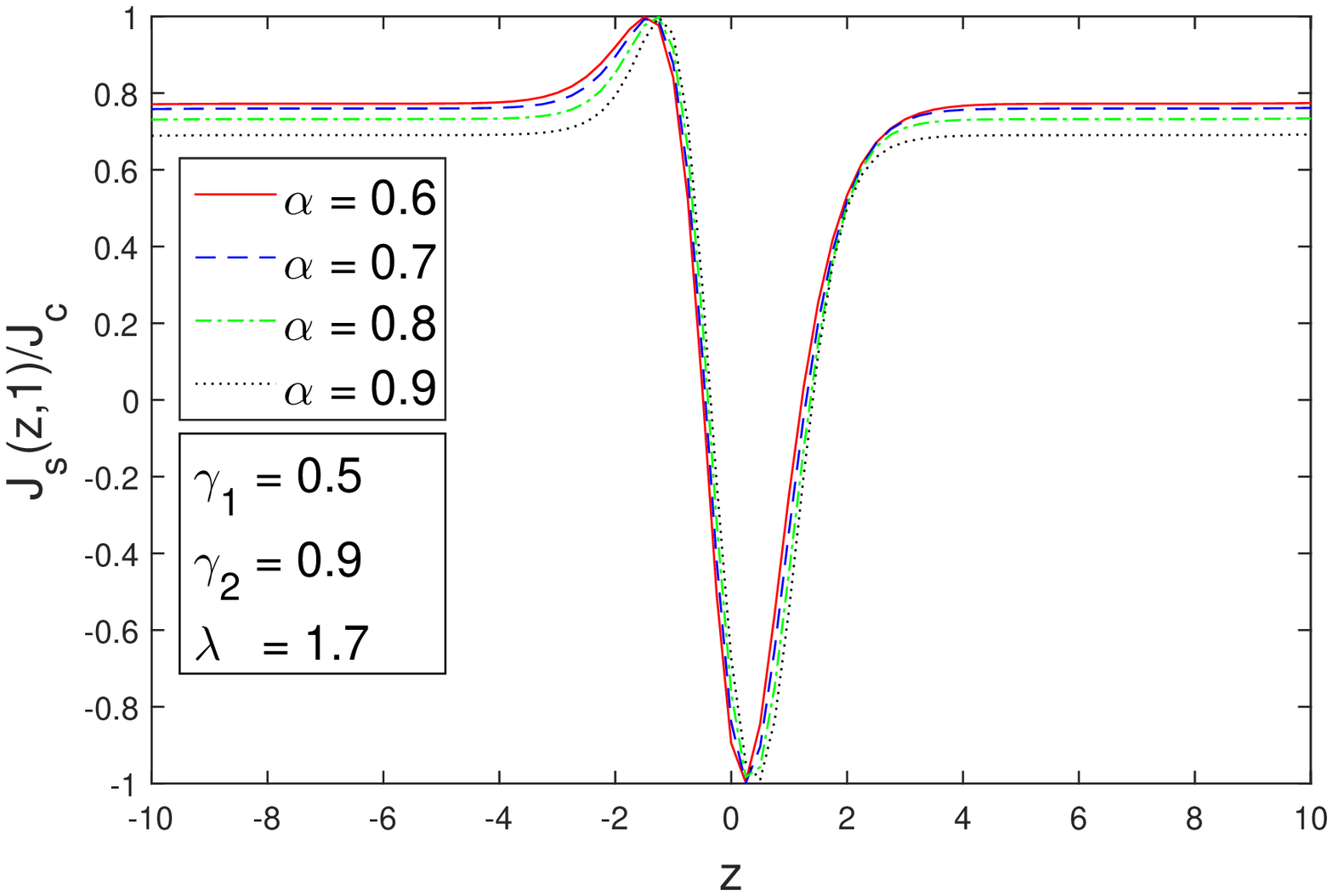}
\label{fig:7.2(b)}} \\
\subfigure[Variation of Voltage with $\alpha$]
{
\includegraphics[width = .45\textwidth, height = .38\textwidth]{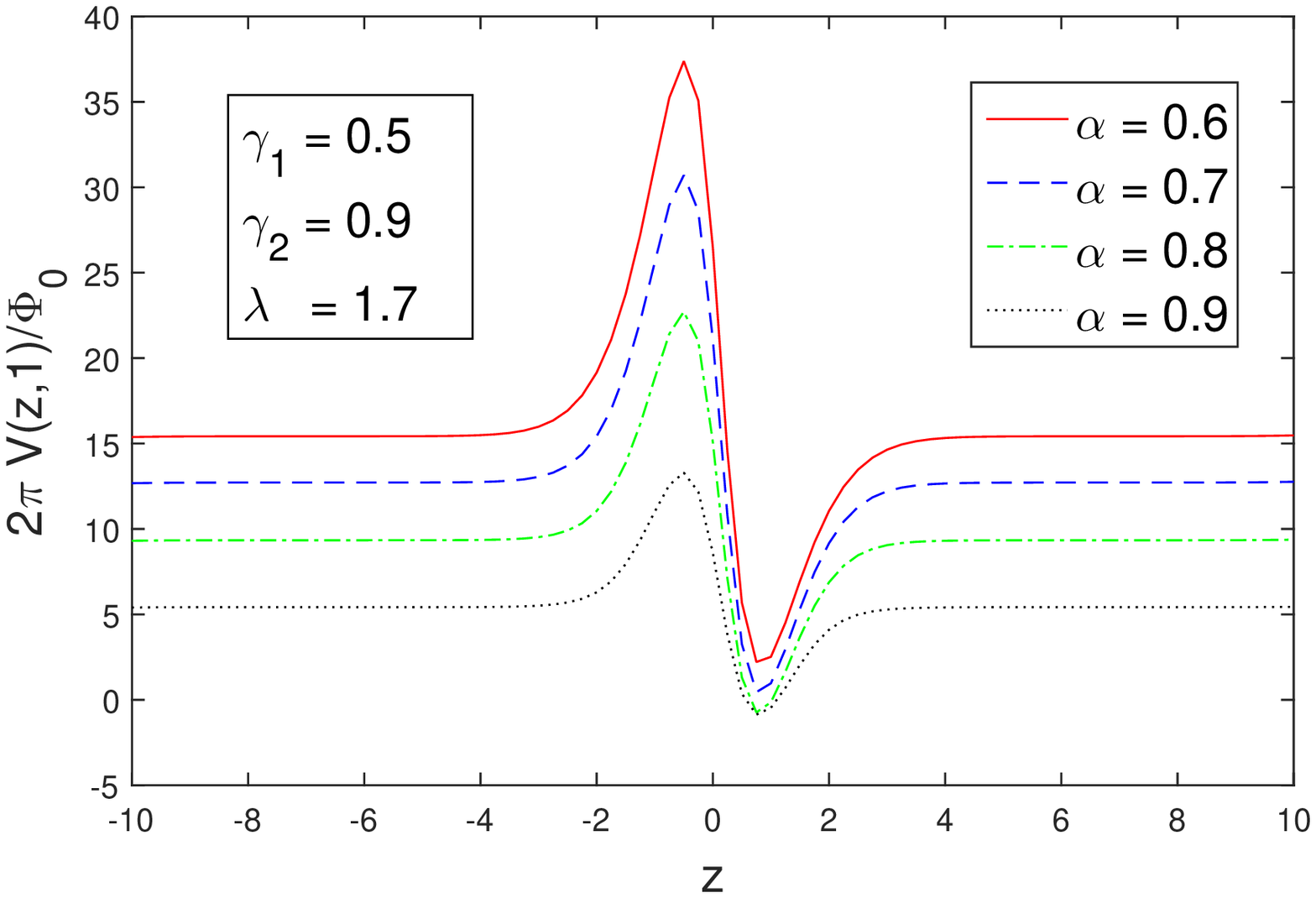}
\label{fig:7.2(c)}}
\caption{Effect of fractional exponent $\alpha$ on Phase difference, Josephson current density and Voltage}
 \label{fig:7.2}
\end{center}
\end{figure}

\begin{figure}[!h]
\begin{center}
\subfigure[Variation of Phase difference
    with ${\gamma}_1$]
{
\includegraphics[width = .45\textwidth, height = .38\textwidth]{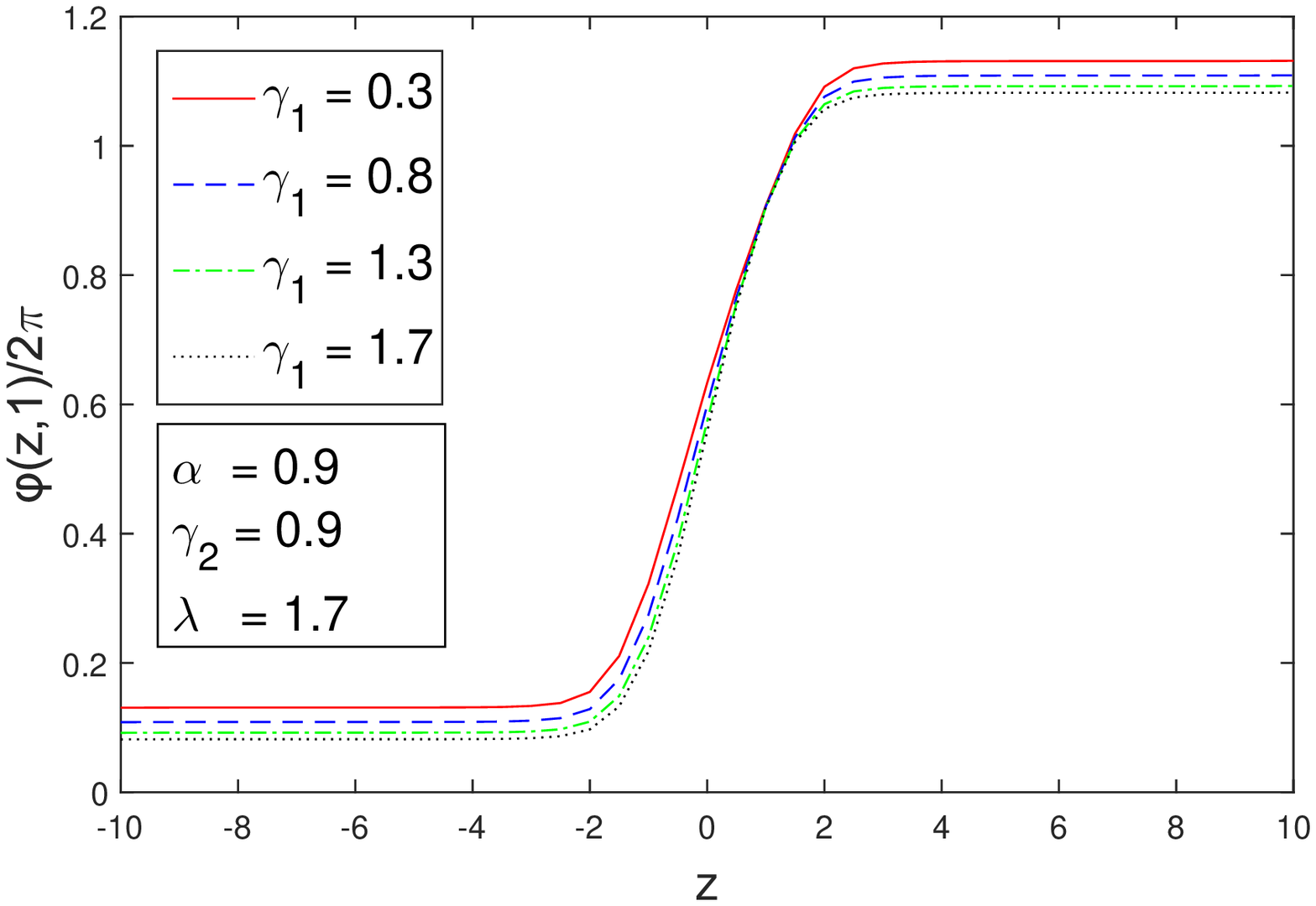}
\label{fig:7.3(a)} }
\subfigure[Variation of Josephson current density with ${\gamma}_1$]
{
\includegraphics[width = .45\textwidth, height = .38\textwidth]{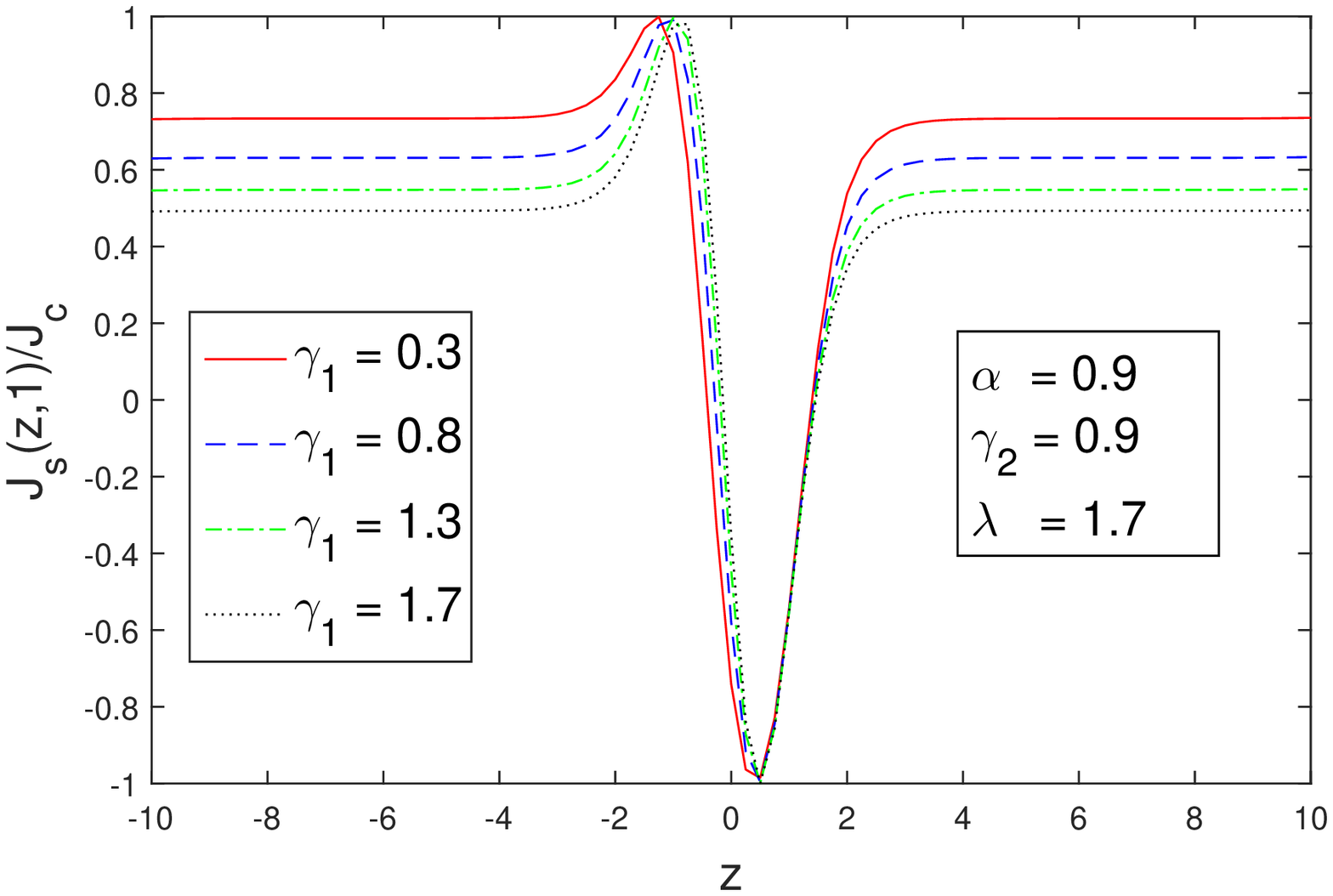}
\label{fig:7.3(b)}}
\subfigure[Variation of Voltage with ${\gamma}_1$]
{
\includegraphics[width = .45\textwidth, height = .38\textwidth]{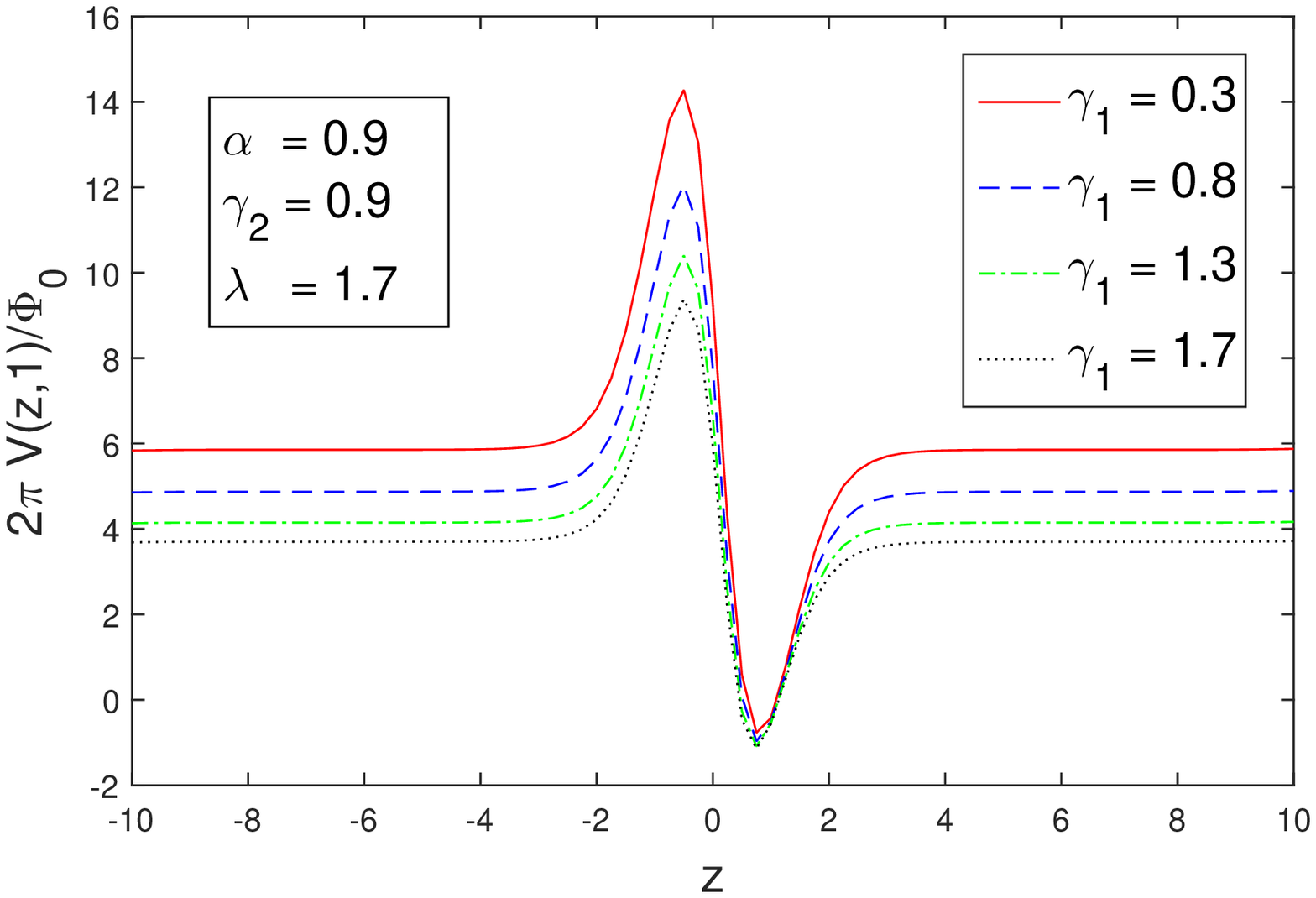}
\label{fig:7.3(c)}}
\caption{Effect of ${\gamma}_1$ on Phase difference, Josephson current density and Voltage}
 \label{fig:7.3}
\end{center}
\end{figure}

\begin{figure}[!h]
\begin{center}
\subfigure[Variation of Phase difference
    with ${\gamma}_2$]
{
\includegraphics[width = .45\textwidth, height = .38\textwidth]{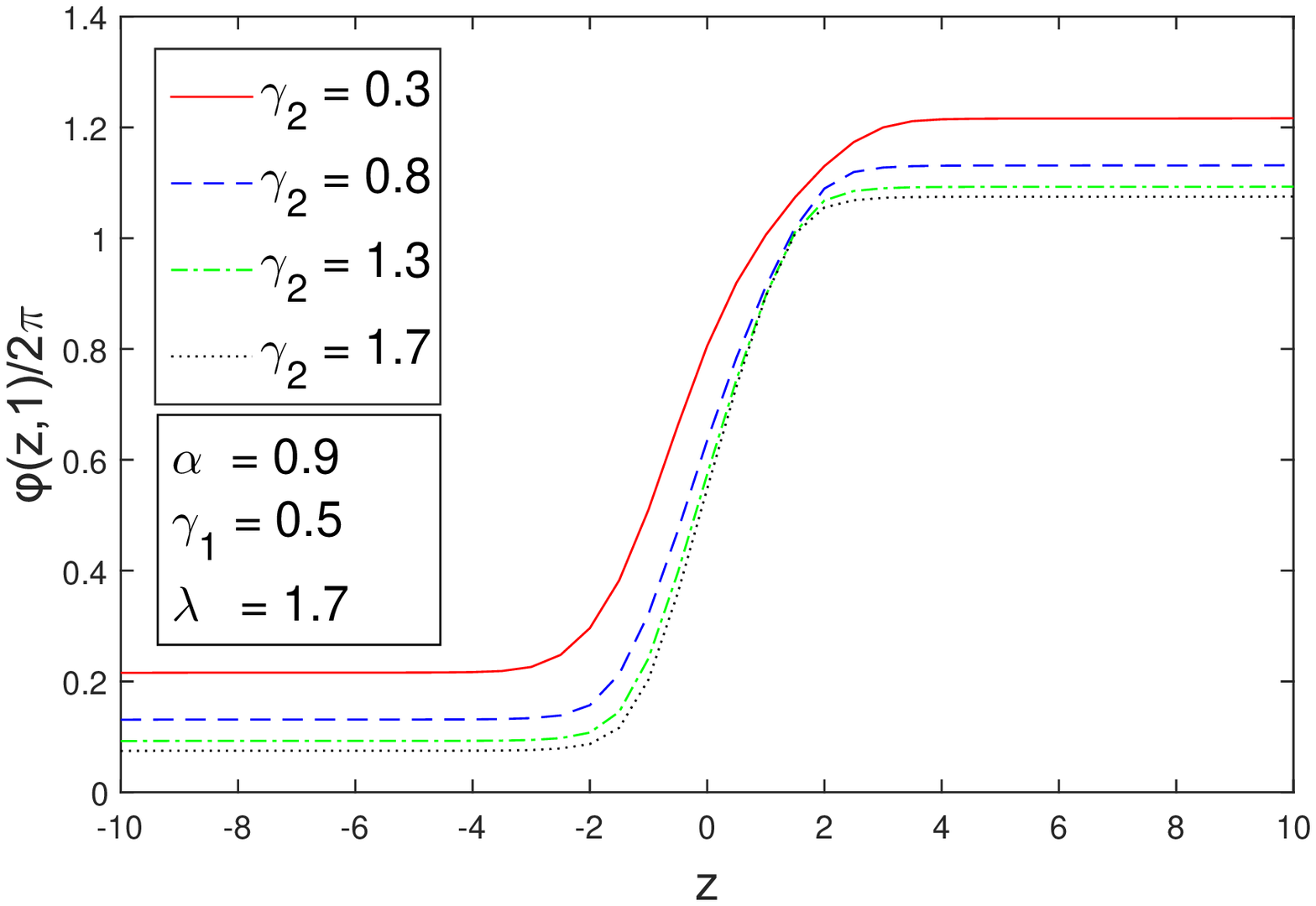}
\label{fig:7.4(a)} }
\subfigure[Variation of Josephson current density with ${\gamma}_2$]
{
\includegraphics[width = .45\textwidth, height = .38\textwidth]{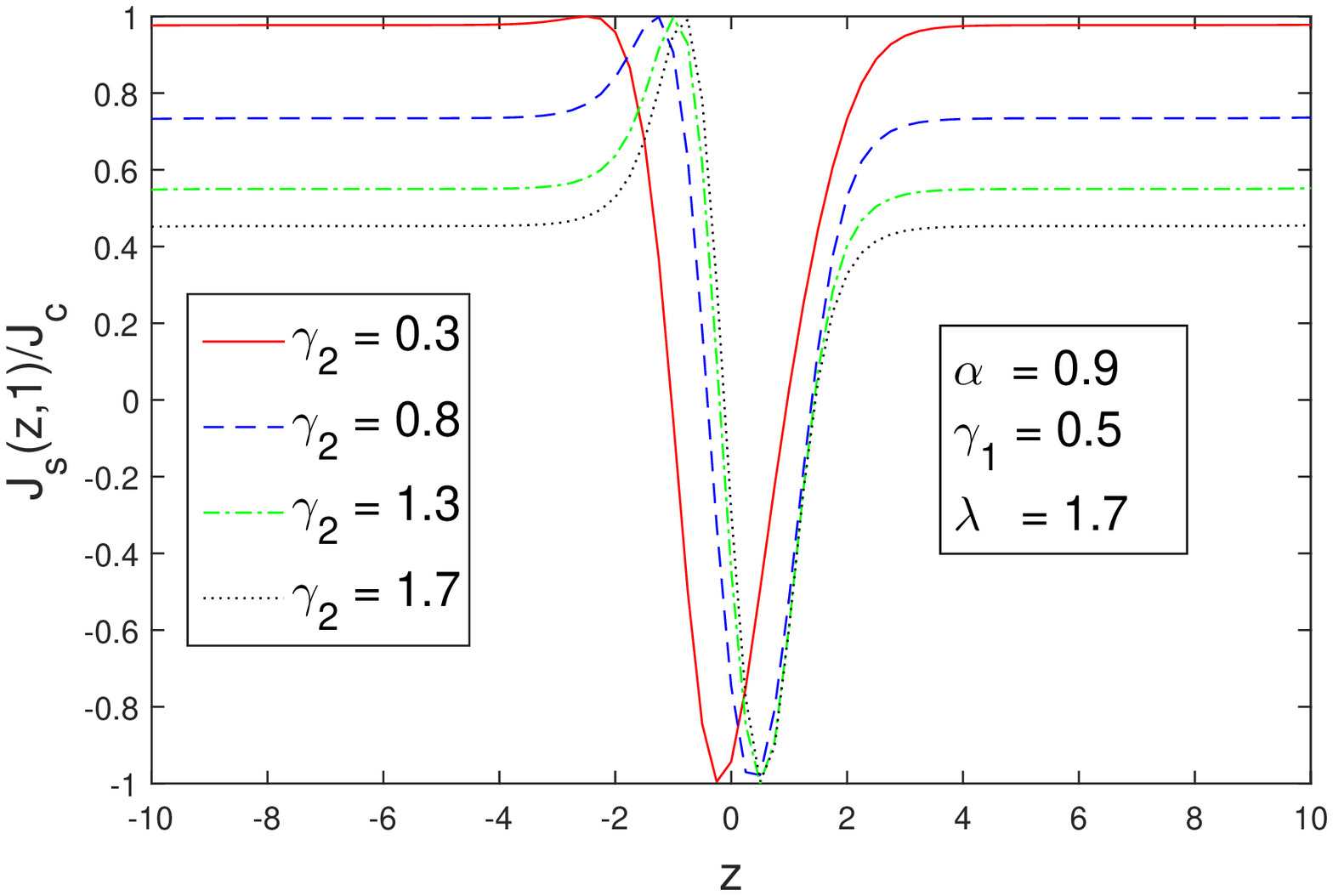}
\label{fig:7.4(b)}}
\subfigure[Variation of Voltage with ${\gamma}_2$]
{
\includegraphics[width = .45\textwidth, height = .38\textwidth]{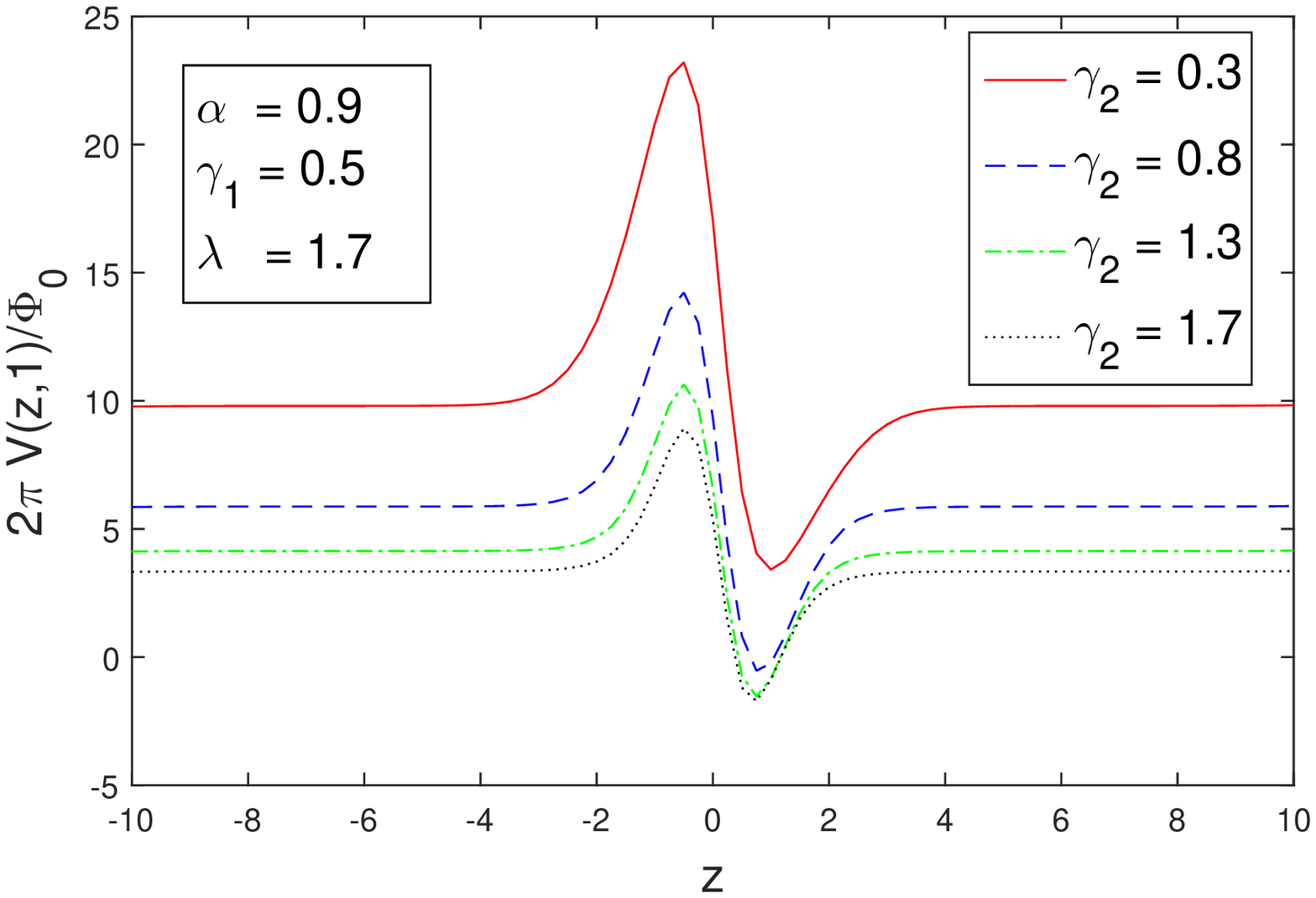}
\label{fig:7.4(c)}}
\caption{Effect of ${\gamma}_2$ on Phase difference, Josephson current density and Voltage}
 \label{fig:7.4}
\end{center}
\end{figure}

\begin{figure}[!ht]
\begin{center}
\subfigure[Variation of Phase difference 
    with $\lambda$]
{
\includegraphics[width = .45\textwidth, height = .38\textwidth]{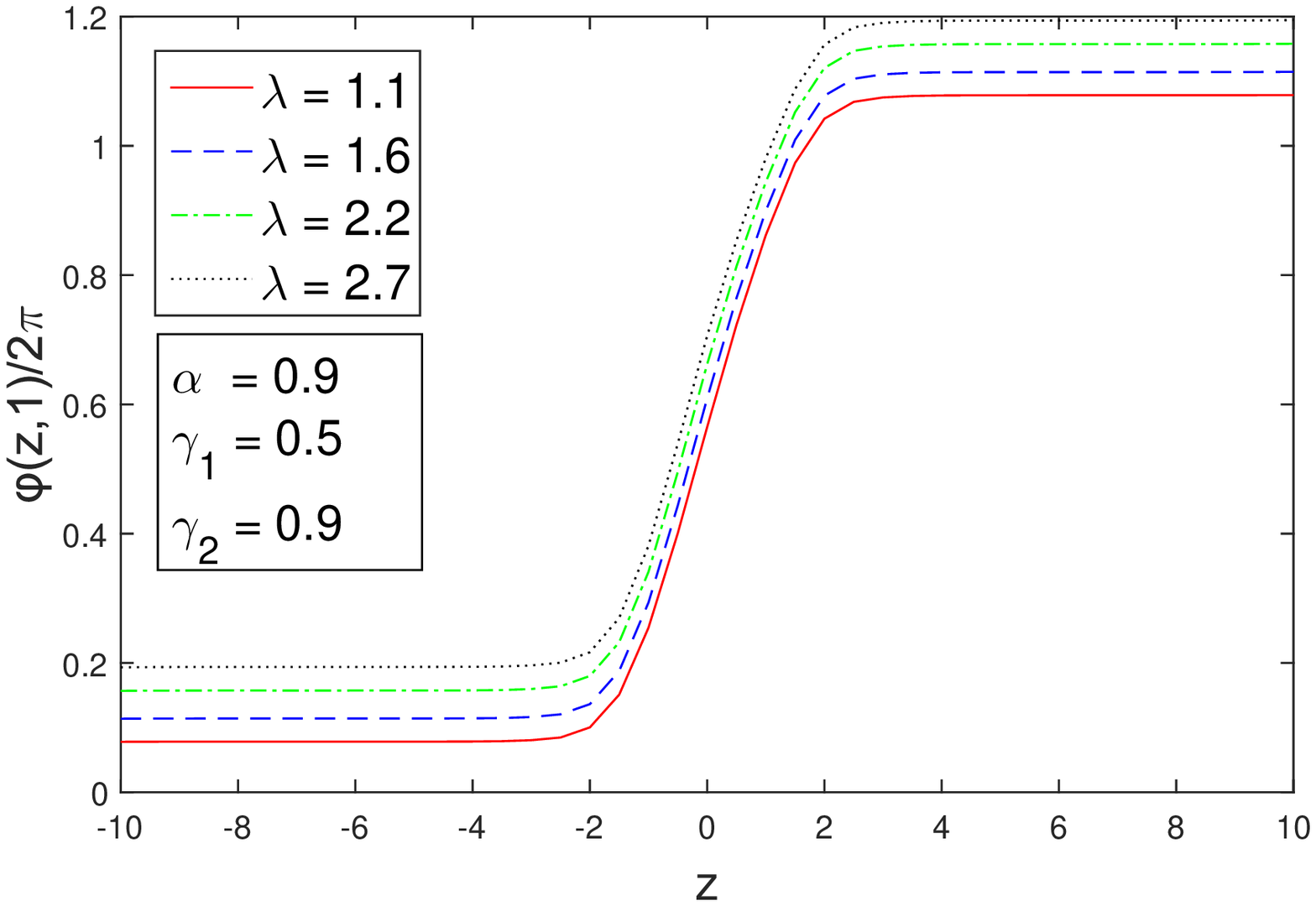}
\label{fig:7.5(a)} }
\subfigure[Variation of Josephson current density with $\lambda$]
{
\includegraphics[width = .45\textwidth, height = .38\textwidth]{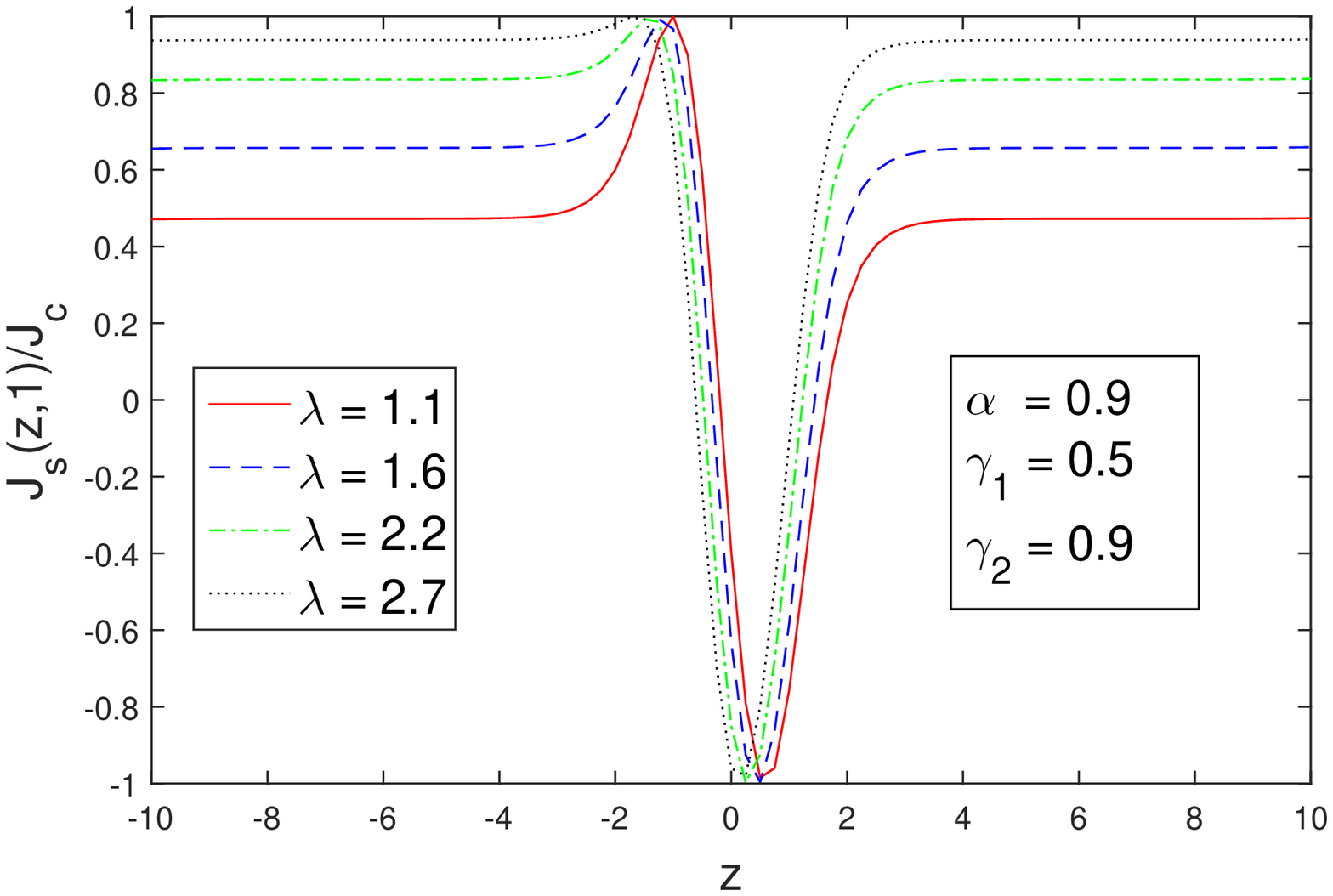}
\label{fig:7.5(b)}}
\subfigure[Variation of Voltage with $\lambda$]
{
\includegraphics[width = .45\textwidth, height = .38\textwidth]{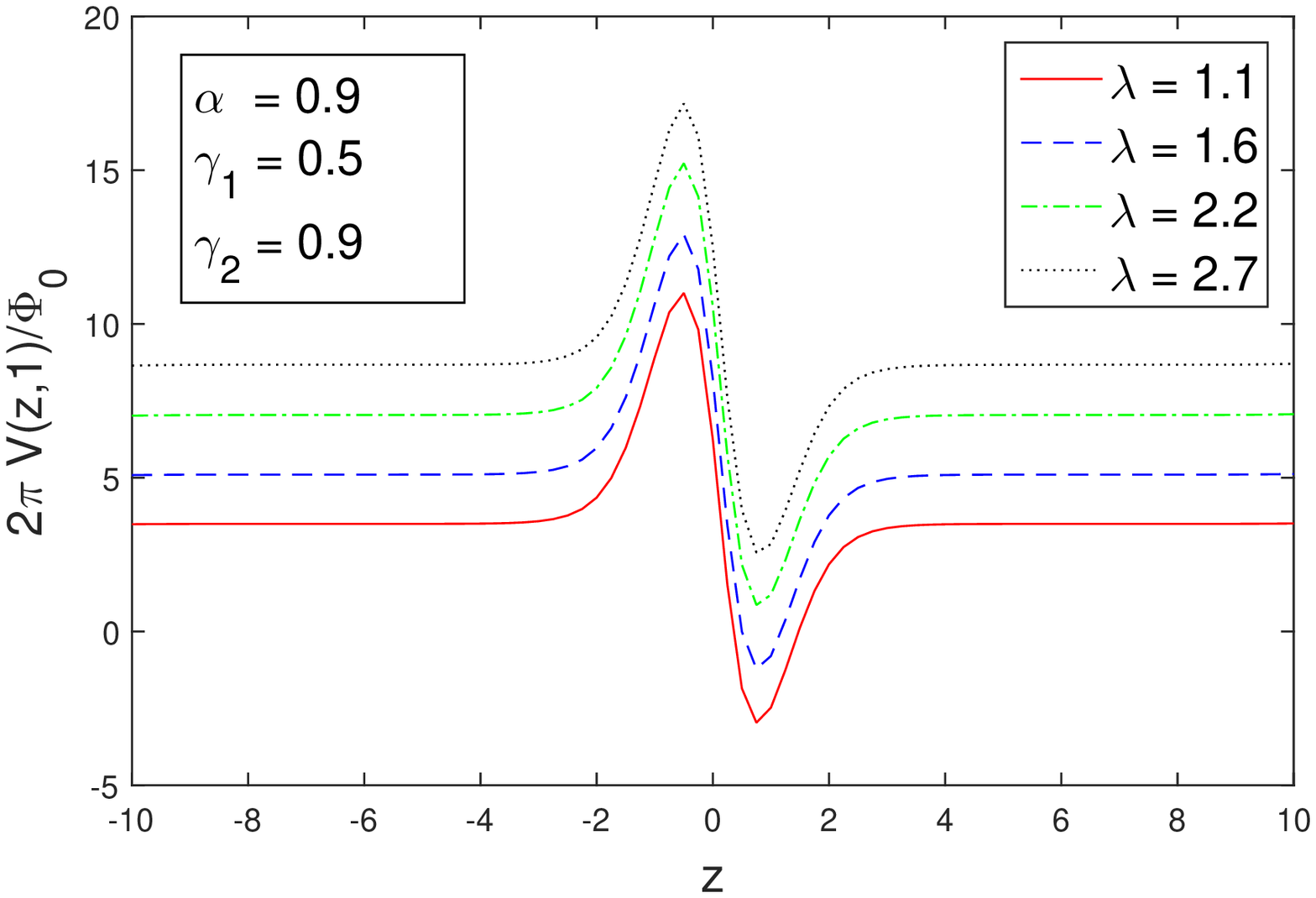}
\label{fig:7.5(c)}}
\caption{Effect of $\lambda$ on Phase difference, Josephson current density and Voltage}
 \label{fig:7.5}
\end{center}]
\begin{center}
\subfigure[]
{
\includegraphics[width = .45\textwidth, height = .38\textwidth]{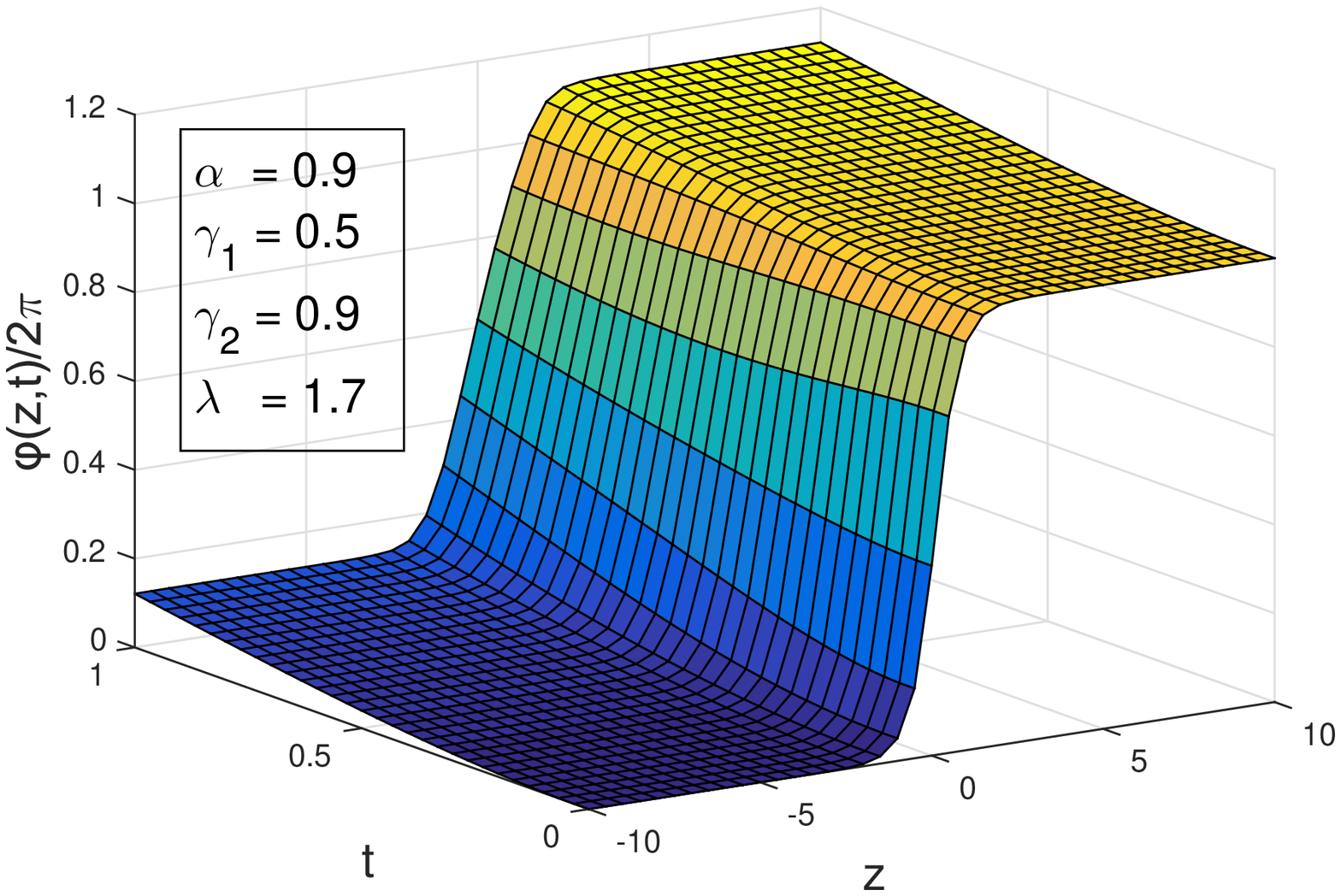}
\label{fig:7.6(a)} }
\subfigure[]
{
\includegraphics[width = .45\textwidth, height = .38\textwidth]{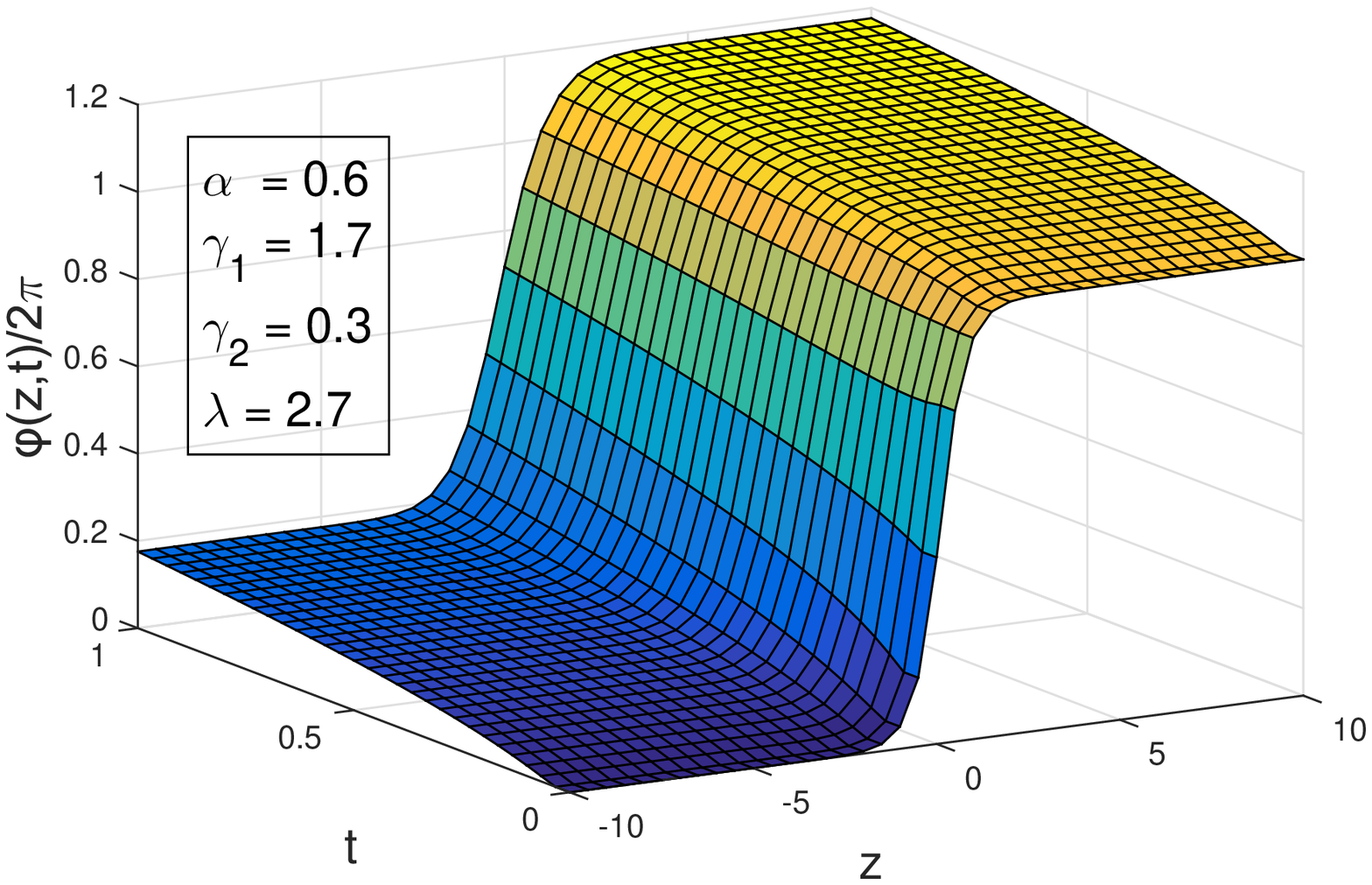}
\label{fig:7.6(b)}}
\caption{Evolution of phase difference for different parameter values over the time interval [0,1]}
 \label{fig:7.6}
\end{center}
\end{figure}

\clearpage
\section{Conclusion} \label{S6}
The evolution of the phase difference between the two superconductors across a long inline Josephson junction, in voltage state and under the influence of magnetic field, is discussed in this article.Josephson current density and voltage across the junction are also calculated from the phase difference.Finite element method along with finite difference method is used to numerically solve the mathematical model describing the evolution of phase difference. Convergence and error analysis of the finite element scheme are also discussed. Effects of different parameters on phase difference, Josephson current density and voltage are studied graphically. Phase difference, Josephson current density and voltage all decrease with increase in $\alpha$ and ${\gamma}_1$. Similar trend is found with increase in ${\gamma}_2$. But with the increase in $\lambda$ phase difference, Josephson current density and voltage increase.\par
The results graphical results obtained for the phase difference, Josephson current density and voltage, using fractional time derivative in the model, are almost similar to the ones we already had for integer order time derivative \cite{r7}, as the fractional exponent $\alpha$ and ${\gamma}_2$ approaches 1.

\clearpage
\bibliographystyle{plain}

\end{document}